\documentclass{amsart} 

  \usepackage{amsmath, amsthm, amsfonts, amssymb, bbm, mathtools} 
  \usepackage{graphicx} 
  \usepackage{mathrsfs} 
  \usepackage{enumitem} 
  \usepackage{xcolor} 
  \usepackage{bm}

  \usepackage{tikz}
  \usepackage{hyperref,csquotes}
  \usepackage[capitalize]{cleveref}

  \newtheorem{theo}{Theorem}[section]%
  \newtheorem{prop}[theo]{Proposition}%
  \newtheorem{coro}[theo]{Corollary}%
  \newtheorem{lemm}[theo]{Lemma}%
  \theoremstyle{definition}%
  \theoremstyle{remark}%
  \newtheorem{rema}[theo]{Remark}%

  \DeclareMathOperator\Law{Law} 
  \DeclareMathOperator\Leb{Leb} 
   
  \DeclareMathOperator\Perm{Perm}
  \DeclareMathOperator\BL{BL}
  \DeclareMathOperator\TL{TL}
  \DeclareMathOperator\exc{\mathfrak e}
  \DeclareMathOperator\Min{Min}
  \DeclareMathOperator\Sample{Sample}
  \DeclareMathOperator\dens{dens}
  \newcommand{ \mc }[1] 			{ \mathcal #1 }
  
  \newcommand{ \mbb }[1] 			{ \mathbb #1 }

  \newcommand{ \ceiling }[1] 		{ \lceil #1 \rceil }

  \newcommand{ \ldef } 				{ \hspace{ 1 pt } \raisebox{ 0.4 pt }{:} \hspace{ -4 pt }= }

  \newcommand{ \eps } 		{ \varepsilon }

  \newcommand{ \indicator } 	{ \mathbbm 1 }
  \newcommand{ \prob } 		{ \mathbb P }
  \newcommand{ \One } 		{ \bm{1}}

  \newcommand{\bdW} {\bm{d_W}}  

  \newcommand\si{\sigma}
  \DeclareMathOperator{\pat}{pat}

  \DeclareMathOperator{\des}{des}
  
  \DeclareMathOperator{\Bin}{Bin}

  \newcommand{\murec}{\bm\mu^{\text{\tiny rec}}}   
    \newcommand{\Wrec}{\bm W^{\text{\tiny rec}}}   
  \newcommand{\muBr}{\bm\mu^{\text{\tiny Br}}}        
    \newcommand{\WBr}{\bm W^{\text{\tiny Br}}}   
  
  \newcommand{\incPermutations}{\lambda}

\begin{document}

  \title[Random recursive separable permutations]	{
    The permuton limit of random recursive separable permutations
          }
\author[{V.} {F\'eray}]{{Valentin} {F\'eray}} 
\address{Universit\'e de Lorraine, CNRS, IECL, F-54000, Nancy, France}
\email{valentin.feray@univ-lorraine.fr (Corresponding author)}

\author[{K.} {Rivera-Lopez}]{{Kelvin} {Rivera-Lopez}}
\address{Department of Mathematics, Gonzaga University, Washington State, USA.}
\email{rivera-lopez@gonzaga.edu}


  \begin{abstract} 
    We introduce and study a simple Markovian model of random separable permutations. Our first main result is the almost sure convergence of these permutations
    towards a random limiting object in the sense of permutons,
    which we call the {\em recursive separable permuton}.
    We then prove several results on this new limiting object:
    a characterization of its distribution via a fixed-point equation,
    a combinatorial formula for its expected pattern densities,
    an explicit integral formula for its intensity measure,
    and lastly, we prove that its distribution is absolutely singular with respect to
    that of the Brownian separable permuton, which is the large size
    limit of {\em uniform} random separable permutations.
  \end{abstract}

\maketitle

\section{Introduction}

\subsection{Our model}
\label{ssec:model}

Fix $p \in (0,1)$.
We consider a sequence of random permutations $( \sigma^{(n),p} )_{n \ge 1 }$ starting from the unique permutation of size $1$ and defined recursively.
Given $\sigma^{(n),p}$, a permutation of size $n$, we obtain $\sigma^{(n+1),p}$, a permutation of size $n + 1$, by the following procedure.
\begin{enumerate}
\item Take $j$ uniformly at random between $1$ and $n$.
\item In the one line notation of $\sigma^{(n),p}$, we increase all values bigger than $j$
by $1$.
\item With probability $p$ (resp. $1-p$), we replace $j$ by $j,\, j\! +\! 1$ (resp. by $j\! +\!1,\, j$).
\end{enumerate}
An example is given on  Figure~\ref{fig:ExamplesUpOperator};
here and throughout the paper, permutations are represented by their diagrams
(the diagram of a permutation $\pi$ of size $n$ is the set of dots $(j,\pi(j))$
 drawn in an $n \times n$ grid).
 This operation of replacing a point in the diagram by two consecutive points
  (consecutive at the same time in value and in position) will be referred to as {\em inflation}.
  We call the inflation {\em increasing} or {\em decreasing},
  depending on the relative position of the two new points.
\begin{figure}
\[\includegraphics[height=32mm]{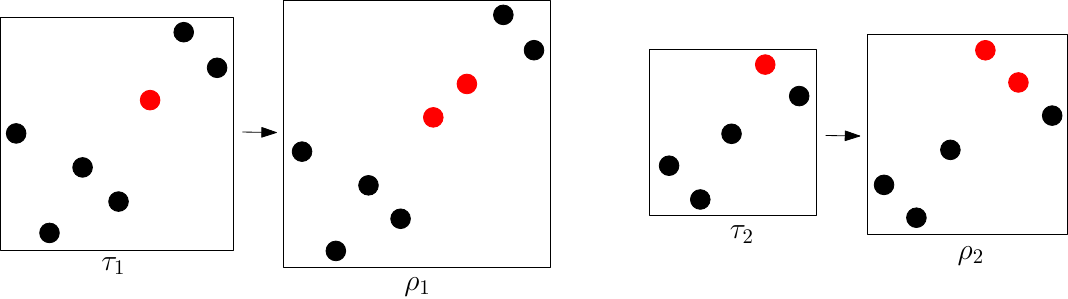}\]
\caption{Examples of possible inflation steps from 
$\tau:=\sigma^{(n),p}$ to $\rho:=\sigma^{(n+1),p}$. For $i=1$ or $2$, 
the point of $\tau_i$ chosen uniformly at random,
as well as the two new adjacent points in $\rho_i$ replacing it, are painted in red.
In $\rho_1$, these two new points are in increasing order
(we say that we have performed an increasing inflation), while,
in $\rho_2$, they are in decreasing order 
(in this case, we have performed a decreasing inflation).}
\label{fig:ExamplesUpOperator}
\end{figure}

The set of permutations which can be obtained from the permutation $1$
by repeated inflations (either increasing or decreasing) is known as the set of {\em separable permutations}. Alternatively, separable permutations are those permutations that avoid the patterns
$3142$ and $2413$,
see, {\em e.g.}, \cite{bose1998pattern}.
By construction, the permutations $( \sigma^{(n),p} )_{ n \ge 1 } $ are separable and any given separable permutation will appear in this sequence with nonzero probability.
We therefore refer to $\sigma^{(n),p}$ as the \emph{random recursive separable 
permutation (of size $n$ and parameter $p$)}.
This model differs from the model of uniform random separable permutations (studied, {\em e.g.}~in \cite{bassino2018BrownianSeparable,pinsky2021separable}) 
and, as we will see in Proposition \ref{prop:singularity}, it yields a different object in the limit.

\begin{rema}
It was asked in a recent survey on permutons 
whether uniform separable permutations can be sampled in a Markovian way
\cite[Section 5.4]{grubel2022ranks}. Though we do not answer this question here,
this served as an additional motivation to study a natural Markovian model of random
separable permutations.
\end{rema}

\subsection{The limiting permuton}

Throughout the paper, let $\Leb$ be the Lebesgue measure on $[0,1]$.
Recall that if $\nu$ is a measure on $A$ and $g$ is a measurable map from $A$ to $B$,
then the formula $g_\#\nu(C)=\nu(g^{-1}(C))$ for any measurable subset $C$ of $B$
defines a measure $g_\#\nu$ on $B$, called {the} {\em push-forward measure}.
Finally, we denote {by} $\pi_1$ and $\pi_2$ the projection map from $[0,1]^2$
to $[0,1]$ on the first and second coordinates, respectively.

By definition, a permuton is a probability measure $\mu$ on the unit square $[0,1]^2$
whose projections on the horizontal and vertical axes are both uniform,
i.e.~$(\pi_1)_\# \mu=(\pi_2)_\# \mu=\Leb$.
Permutons are natural limit objects for permutations of large sizes,
see~\cite{grubel2022ranks} for a recent survey on the topic.
Indeed, we can encode a permutation $\pi$ of size $n$
by a permuton $\mu_{\pi}=\frac1n \sum_{i=1}^n \lambda(i,\pi(i))$,
where $\lambda(i,\pi(i))$ is the measure of mass $1$ uniformly spanned
on the square $[\frac{i-1}{n};\frac in ] \times [\frac{\pi(i)-1}{n};\frac{\pi(i)}{n}]$.
Equivalently, $\mu_\pi$ has a piecewise constant density 
\[g(x,y)=\begin{cases}
  n & \text{ if } \pi(\lceil nx \rceil)= \lceil ny \rceil;\\
  0 & \text{ otherwise.}
\end{cases}\]
A sequence of permutations $\pi^{(n)}$ then converges to a permuton $\mu$
if the associated measures $\mu_{\pi^{(n)}}$ converge to $\mu$,
in the sense of the weak convergence of measures.
We can now state the first main result of this paper.
\begin{theo}\label{thm:convergence}
The random permutations $ ( \sigma^{(n),p} )_{ n \ge 1 } $ converge a.s.~to 
a random permuton.
We call this permuton the {\em recursive separable permuton (of parameter $p$)} and denote it by $\murec_p$.
\end{theo}
A sample of $\sigma^{(n),p}$ for $n \in \{10,100,1000\}$ and $p=1/2$
is given in Figure~\ref{fig:simu_increasing_separable_several_sizes}.
In this simulation, the three permutations are taken 
from the {\em same realization} of the random process $(\sigma^{(n),p})_{n \ge 1}$.
The a.s.~convergence towards a limiting permuton is visible
on this simulation.

\begin{figure}
\[\includegraphics[height=2.9cm]{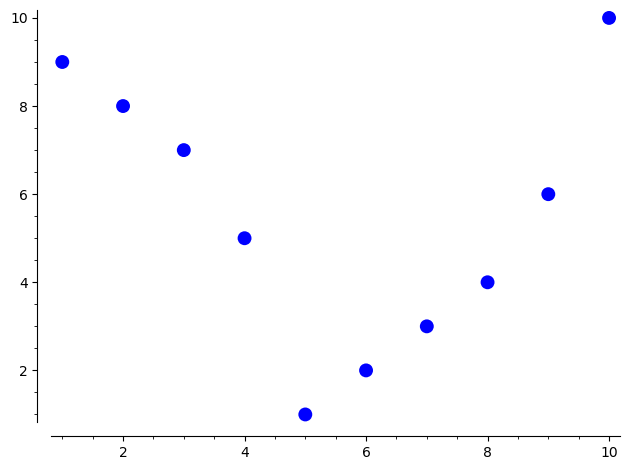}
\quad  \includegraphics[height=2.9cm]{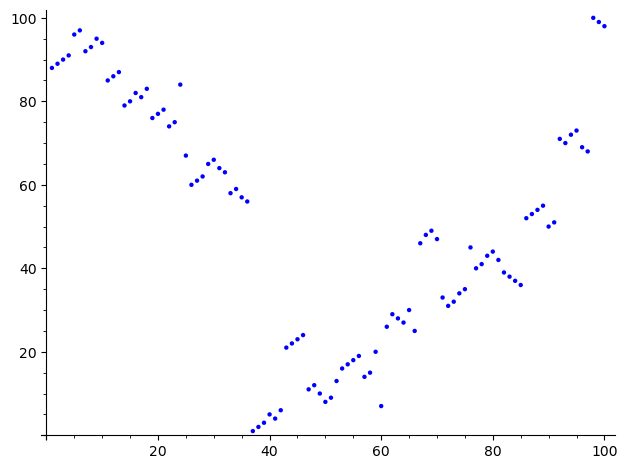}
\quad  \includegraphics[height=2.9cm]{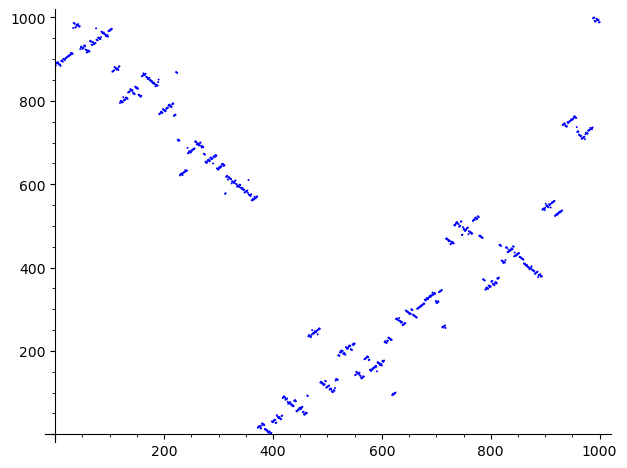}\]
\caption{A sample of permutations $\sigma^{(10),{\, 0.5}}$, $\sigma^{(100),{\, 0.5}}$,
$\sigma^{(1000),{\, 0.5}}$ corresponding to the same realization of the process 
$(\sigma^{(n),{\, 0.5}})_{n \ge 1}$.}
\label{fig:simu_increasing_separable_several_sizes}
\end{figure}

This theorem should be compared with the limit result
for {\em uniform} random separable permutations obtained in
\cite{bassino2018BrownianSeparable}.
In the latter case, the limit is the so-called {\em Brownian separable permuton}.
This Brownian separable permuton belongs to the larger family
of biased Brownian separable permutons, also indexed by a parameter $p$ in $(0,1)$,
and denoted $\muBr_p$ (we will often drop the word {\em biased} in the latter).

The constructions of the recursive and Brownian separable permutons 
share some similarities, but these are different distributions on the set of permutons.
In fact, we prove in Section~\ref{ssec:Singular_Distributions} the following stronger statement.
\begin{prop}\label{prop:singularity}
Let $p \ne q$ be fixed in $(0,1)$.
Then the distributions of the four permutons $\murec_p$, $\murec_q$, $\muBr_p$
and $\muBr_q$ are pairwise singular\footnote{We recall that two measures $\mu$ and $\nu$
are singular if there exists a measurable set $A$ such that $\mu(A)=1$ but $\nu(A)=0$.}.
\end{prop}
Thus our model of increasing separable permutations
yields a much different limiting object than the uniform separable permutations.
This is in contrast with the result of \cite{bassino2020universal},
where the Brownian separable permuton was shown to be the limit
of uniform random permutations in many permutation classes.
It seems that going from the uniform model to a recursive model as the one studied
here allows to escape the universality class of the Brownian separable permuton
(see Remark~\ref{rm:uniform_vs_recursive} for further discussion).

Another difference with the setting of uniform permutations
is that the convergence in Theorem~\ref{thm:convergence} holds in the almost sure sense.
In particular, we cannot use the criterium stating that convergence of permutons
in distribution is equivalent to the convergence of expected pattern densities.
Instead, we need to construct the limiting permuton on the same probability
space as the process $(\sigma^{(n),p})_{n \ge 1}$ of random permutations
and to prove the convergence with adhoc arguments.

\begin{rema}
  \label{rm:uniform_vs_recursive}
{There are other cases in the literature of combinatorial objects
for which the uniform measure and a natural recursive random construction
yield different asymptotic behaviors.
Most famously, for many families of trees, the typical height
of a random tree is of order $\sqrt n$ in the uniform model \cite{aldous1993CRT3}
and $\log(n)$ in recursive models \cite{pittel1994recursive_trees}.
A situation closer to the one of this paper in which we
have nontrivial but different limits for the uniform and recursive models
is that of noncrossing sets of chords in a regular $n$-gon.
Indeed, in~\cite{curien2009recursive}, Curien and LeGall
consider recursive models of noncrossing sets of chords
and show that they converge to a random limiting object $L_\infty$.
This random set $L_\infty$ is different from the limit of uniform
random triangulations of the $n$-gon, previously identified by Aldous
and known as the Brownian triangulation \cite{aldous1994triangulating}.
Let us also mention recent results on random graphs with fixed degree sequences
and random chirotopes (chirotopes are combinatorial objects encoding the relative positions of a set of points in the plane), which both exhibit differences in the asymptotic behaviors
between recursive and uniform models \cite{goaoc2020convex,molloy2022degreerestricted}.
We do not have a satisfactory intuitive explanation of these facts.}
\end{rema}

\begin{rema}
  In the theory of permutons (see, e.g., \cite[Lemma 4.2]{MR2995721}),
  there is an explicit construction for permutations that converge almost surely to a given permuton.
  It is natural to compare the permutations obtained by applying this construction to the random permuton $\murec_p$ with the recursive separable permutations.
  As sequences, these two objects are clearly different: the recursive separable permutations are constructed via repeated inflations while the other sequence is not.
  However, we shall see in Section~\ref{ssec:expected_pattern_densities} that these sequences do have the same marginal distributions.
\end{rema}

\subsection{Properties of the recursive separable permuton}
\subsubsection{Self-similarity}
The distribution of the recursive separable permuton $\murec_p$
can be characterized by a fixed-point equation. To state this property, 
we first need to introduce some notation.

Given two permutons $\mu$ and $\nu$, a real number $u$ in $[0,1]$ and a sign $S$ in $\{\oplus,\ominus\}$, we construct a new permuton $\rho=\mu \otimes_{(u,S)} \nu$ as follows.
If $S=\oplus$, we let $\rho$ be 
supported on $[0,u]^2 \cup [u,1]^2$ such that the restriction 
$\rho\vert_{ [0,u]^2 }$ (resp.~$\rho\vert_{[u,1]^2}$) 
is a rescaled version of $\mu$ of total weight $u$
(resp.~a rescaled version of $\nu$ of total weight $1-u$).
If $S=\ominus$, then $\rho$ is defined similarly,
but is 
supported on $[0,u]\times [1-u,1] \cup [u,1] \times [0,1-u]$.
We refer to Figure~\ref{fig:SelfSimilarityRecusivePermuton} for an illustration.

\begin{figure}
  \begin{center}
    \includegraphics[scale=.85]{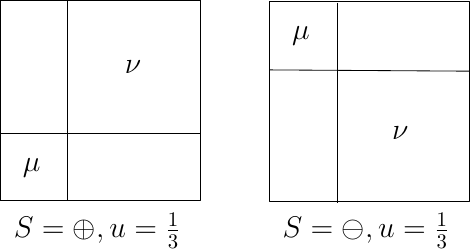}\vspace{-4mm}
  \end{center}
  \caption{The permuton $\mu \otimes_{(u,S)} \nu$.}
  \label{fig:SelfSimilarityRecusivePermuton}
\end{figure}

Now, given a random permuton $\bm\mu$,
we denote by $\Phi_p(\bm\mu)$ the random permuton 
$\bm\mu_0 \otimes_{(U,S)} \bm\mu_1$,
where $\bm\mu_0$ and $\bm\mu_1$ are two  copies of $\bm\mu$,
 $U$ is a uniform r.v.~in $[0,1]$ and $S$ is a random sign in $\{\oplus,\ominus\}$
with $\mathbb P(S=\oplus)=p$, all variables $\bm\mu_0$, $\bm\mu_1$, $U$ and $S$
being independent.

\begin{prop}
\label{prop:self_similarity}
For any $p$ in $[0,1]$, we have $\murec_p \stackrel{d}{=}\Phi_p(\murec_p)$.
Moreover, the distribution of $\murec_p$ is characterized by this property,
in the following sense: if a random permuton $\nu$ satisfy $\nu \stackrel{d}{=}\Phi_p(\nu)$,
then $\nu \stackrel{d}{=}\murec_p$.
\end{prop}

\subsubsection{Expected pattern densities}
\label{ssec:pattern_densities_intro}
We recall that if $\pi$ is a pattern (i.e.~a permutation) of size $k$
and $\mu$ a permuton,
we can define the random permutation $\Sample(\mu;k)$
and the pattern density $\dens(\pi,\mu)$ of $\pi$ in $\mu$ as follows.
Let $(x_i,y_i)$ be i.i.d.~points in $[0,1]^2$ with distribution $\mu$.
We reorder them as $(x_{(1)},y_{(1)})$,\dots, $(x_{(k)},y_{(k)})$
such that $y_{(1)} < \dots < y_{(k)}$.
Then there exists a unique (random) permutation $\tau$ such that
$ x_{(\tau_1)} < \dots < x_{(\tau_k)}$.
This random permutation $\tau$ is denoted $\Sample(\mu;k)$.
 We also write 
 $$\dens(\pi,\mu)=\mbb P\big(\Sample(\mu;k)=\pi\big).$$

These functionals plays a key role in the theory of permutons.
In particular, convergence of permutons is equivalent to convergence of all pattern
densities, see \cite{MR2995721}.
Also the distribution of a random permuton is uniquely determined
by its expected pattern densities \cite[Proposition~2.4]{bassino2020universal}. 
The next proposition provides combinatorial descriptions of these
expected pattern densities in the case of the recursive separable permuton.

To state it, we introduce some terminology.
Let $\pi$ and $\sigma$ be two permutations of respective sizes $k$ and $\ell$. 
Their \emph{direct sum} and \emph{skew sum} 
are the permutations of size $k + \ell$ defined in one-line notation as follows
\begin{align*}
\pi \oplus \sigma = & \pi_1 \ldots \pi_k (\sigma_1 +k) \ldots (\sigma_\ell +k) ; \\
 \pi \ominus \sigma = & (\pi_1+\ell) \ldots (\pi_k+\ell) \sigma_1 \ldots \sigma_\ell. 
\end{align*}
Examples with $k=3$ and $ \ell=2$ are provided in Figure~\ref{fig:sum_and_skew}.
They illustrate the graphical interpretation of these operations
on permutation diagrams.
\begin{figure}[t]
    \[
      \includegraphics[width=6cm]{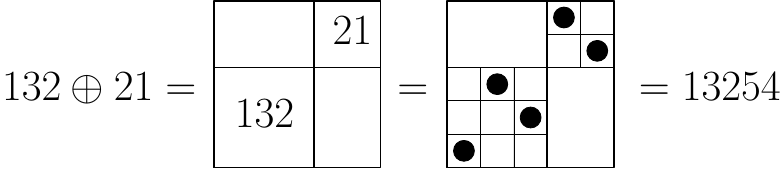} \quad 
      \includegraphics[width=6cm]{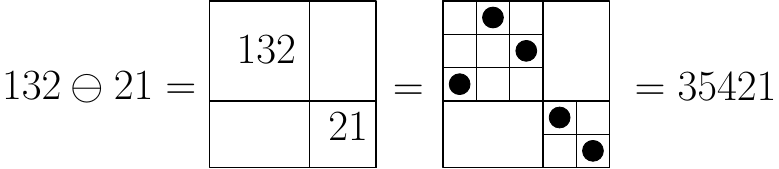}
    \]
\caption{Direct sum and skew sum of permutations.
\label{fig:sum_and_skew}}
\end{figure}

In the following, we consider rooted (complete) binary trees,
meaning that every internal node has exactly two ordered children.
Additionally, internal nodes are labeled with numbers from $1$ to some $k$ 
(where each integer in this range is used exactly once). 
The tree is said to be increasing
if labels are increasing on any path from the root to a leaf.
Finally, each internal node carries a decoration, which is either $\oplus$ or $\ominus$.
To such a tree $T$, we associate a permutation $\si=\Perm(T)$ as follows 
(an example of a rooted increasing binary tree and the associated permutation
is given in Figure~\ref{fig:tree_to_permutation}).
\begin{figure}[t]
\[      \includegraphics[scale=.7]{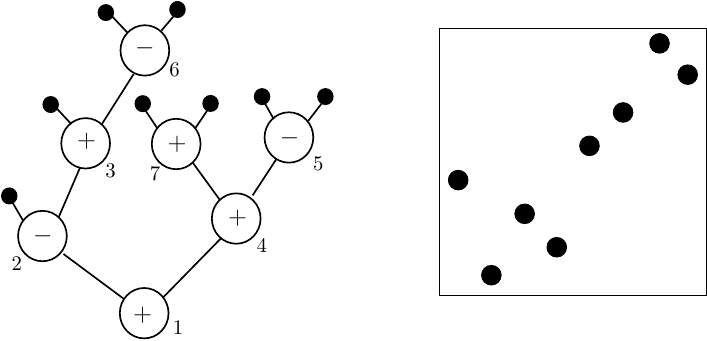}  \]
\caption{A rooted increasing binary tree and the associated permutation.
The permutation associated to the left subtree of the root is $4132$,
the one associated to its right-subtree $1243$. Since the root has decoration $\oplus$,
the permutation associated to the whole tree is $4132 \oplus 1243$,
which is equal to $41325687$.
\label{fig:tree_to_permutation}}
\end{figure}
\begin{itemize}
  \item If $T$ is reduced to a single leaf, then $\si$ is the one-element permutation $1$.
  \item Otherwise, the root of $T$ has two (ordered) children and we call $T_1$ and $T_2$ 
    the subtree rooted at these children. Let $\si_1$ and $\si_2$ be the permutations
    associated with $T_1$ and $T_2$.
    Then we associate $\si_1 \otimes \si_2$ with $T$,
    where $\otimes$ is the decoration of the root of $T$ ($\otimes  \in \{\oplus,\ominus\}$).
\end{itemize}
We note that $\Perm(T)$ does not depend on the labeling of the internal nodes of $T$
(only on the shape of $T$ and on the decorations of its internal nodes).
By construction, $\Perm(T)$ is always a separable permutation.
A separable permutation $\si$ is in general associated with more than one tree $T$.
We note that, except for the labeling of internal nodes of $T$,
 this is a standard construction in the theory of separable permutations,
see, e.g., \cite{bose1998pattern}.

\begin{prop}
\label{prop:expected_densities_murec}
For any pattern $\pi$ of size $n$,
we have
\[\mbb E\big[\dens(\pi,\murec_p) \big] = \mbb P\big[\sigma^{(n),p} =\pi \big] = \frac{N_{inc}(\pi)}{(n-1)!}\, (1-p)^{\des(\pi)}\, p^{n-1-\des(\pi)},  \]
where $\des(\pi)$ is the number of descents in $\pi$ and $N_{inc}(\pi)$
the number of increasing binary trees $T$ such that $\Perm(T)=\pi$.
\end{prop}
If $\pi$ is not a separable pattern, then  $N_{inc}(\pi)=0$, implying
$\dens(\pi,\murec_p)=0$ a.s.

\begin{rema}
  The fact that $\dens(\pi,\murec_p)=0$ a.s.~for non-separable patterns $\pi$ implies that the distribution of $\murec_p$ is also singular with respect
  to that of the so-called skew Brownian permutons $\mu_{\rho,q}$ with parameter $(\rho,q)$ in $(-1,1) \times (0,1)$.
  Indeed, the latter satisfy $\dens(\pi,\murec_p)>0$ a.s., see \cite[Theorem 1.10]{borga2023baxter}.
\end{rema}

\subsubsection{The intensity measure}
Permutons are measures, so that random permutons are random measures.
Given a random measure $\bm\mu$, one can define its {\em intensity measure} $I \bm\mu$,
sometimes also denoted $\mathbb E \bm\mu$ as follows:
for any measurable set $A$ of the ground space, we have
$I \bm\mu (A) = \mathbb E [\bm\mu (A)]$.

Our next result is a simple description of the intensity measure of the random 
permuton $\murec_p$ in terms of beta distributions.
Recall that the distribution $\beta(a,b)$ with positive parameters $a$ and $b$ is given by
\[\frac{\Gamma(a+b)}{\Gamma(a)\, \Gamma(b)} x^{a-1} (1-x)^{b-1} \, dx.\]
\begin{prop}
\label{prop:intensity_beta}
The intensity measure $I \murec_p$ of the recursive separable permuton
is the distribution of
\[ (U, \, U X_p +(1-U) X'_p),\]
where $U$, $X_p$ and $X'_p$ are independent random variables in $[0,1]$,
with distribution $\Leb$,
$\beta(p,1-p)$ and $\beta(1-p,p)$ respectively.
\end{prop}
From this, we can get an explicit formula for the density of $I \murec_p$.
 \begin{coro}
 \label{corol:intensity_density}
 $I \murec_p$ is absolutely continuous
with respect to Lebesgue measure on $ [0, 1]^2 $
and has density
\[\frac{1}{\Gamma(p)^2\Gamma(1-p)^2} \,
\int_{\max(x+y-1,0)}^{\min(x,y)}\frac{dz}{z^{1-p}(x-z)^p(y-z)^p(1-x-y+z)^{1-p}}. \]
 \end{coro}
Fig.~\ref{fig:plot_intensity} shows 3D plots of the density
of  $I \murec_p$ for $p=0.5$ and $p=0.6$ (obtained with Mathematica).
We note that the density diverges for $x=y$ when $p \ge 1/2$ and for $x+y=1$ when $p \le 1/2$. This is different from the Brownian separable permuton case
where the density of the intensity measure diverges only in the corners (i.e.~when both $x$ and $y$ are $0$ or $1$; see \cite{maazoun2020BrownianPermuton} for an explicit formula for the density
in this case).
\begin{figure}
\[ \includegraphics[height=4cm]{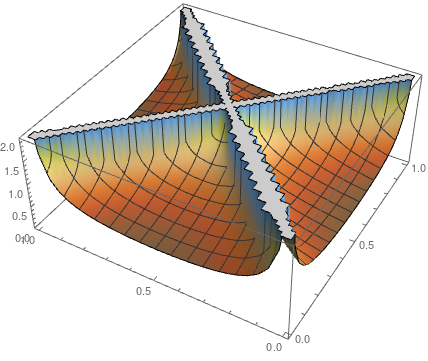} \qquad  \includegraphics[height=4cm]{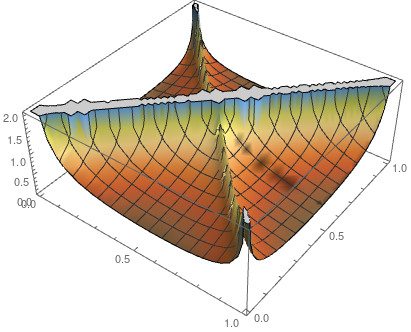} \]
\caption{3D plots of the density
of  $I \murec_p$ for $p=0.5$ (left) and $p=0.6$ (right).}
\label{fig:plot_intensity}
\end{figure}
\bigskip

We conclude this paragraph with a discussion presenting
the intensity measure of the recursive separable permuton
as a limit of a natural discrete object.
This interpretation is a motivation for computing the intensity measure,
and is not needed later in the article. 
Note also that this is not specific to the recursive separable permutation;
a similar discussion could be made for other models of permutations converging to 
random permutons.

For $n \ge 1$ and $N \ge 1$, let $\sigma^{(n),p}_1$, \dots,
$\sigma^{(n),p}_N$ be independent copies of the recursive random separable permutation $\sigma^{(n),p}$.
We then consider the average of the associated permutons:
\begin{equation}\label{eq:discrete_intensity}
  \mu^{(n),p}_N = \frac1N \sum_{i=1}^N \mu_{\sigma^{(n),p}_i}.
\end{equation}
This is a measure on $[0,1]^2$ with piecewise constant density
\[  g^{(n),p}_N(x,y) = \frac{n}N \sum_{i=1}^N \indicator \Big[ \sigma^{(n),p}_i\big(\lceil nx \rceil\big) = \lceil ny \rceil \Big].\]
When $N$ tends to infinity, $\mu^{(n),p}_N$ converges to $I\, \mu_{\sigma^{(n),p}}$,
the intensity measure of the random permuton associated to $\sigma^{(n),p}$.
This measure $I\, \mu_{\sigma^{(n),p}}$ in turn converges to $I\, \murec_p$ as $n$ tends to $+\infty$,
as a consequence of Theorem~\ref{thm:convergence}.
Therefore, for large $n$ and $N$ with $N \gg n$,
the empirical average measures $\mu^{(n),p}_N$ defined in \eqref{eq:discrete_intensity} 
can be seen as a discrete approximation of $I\, \murec_p$.
On Figure~\ref{fig:discrete_intensity}, we plot the density $g^{(n),p}_N$ for $n=200$, 
$N=10000$ and $p \in \{.5,.6\}$. The convergence to that of $I\, \murec_p$
(Fig.~\ref{fig:plot_intensity}) is plausible on the pictures.
\begin{figure}
  \begin{center}
   \[ \includegraphics[height=4cm]{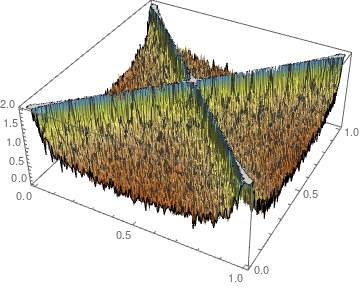} \qquad  \includegraphics[height=4cm]{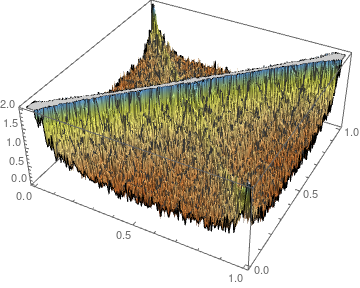} \]
  \end{center}
  \caption{3D plots of the density $g^{(n),p}_N$
  for $n=200$, 
$N=10000$ and $p$ being either $0.5$ (left) or $0.6$ (right).}
  \label{fig:discrete_intensity}
\end{figure}

\subsection{Analog results for cographs}
With a permutation $\sigma$ of size $n$,
it is standard to associate its inversion graph $G_\sigma$ on vertex set $\{1,\dots,n\}$.
By definition, $\{i,j\}$ is an edge of $G_\sigma$ if and only if it is an inversion in $\sigma$,
{\it i.e.}~\hbox{$(i-j)(\sigma(i)-\sigma(j))<0$}.
Inversion graphs of separable permutation are called {\em cographs}.
Cographs can alternatively be described as graphs avoiding the path $P_4$ on four vertices as induced subgraph, 
or as graphs that can be obtained starting from single
vertex graphs and iterating \enquote{disjoint union} and 
\enquote{taking the complement} operations.
We refer to the introduction of \cite{bassino2022cographs} for more background on cographs.

Uniform random cographs have recently been studied in
\cite{bassino2022linear,bassino2022cographs,stufler2021cographs}.
Considering the inversion graphs of  random recursive separable permutations
yields a natural Markovian model of random cographs.
It can be described directly on graphs, without going through permutations.

Namely,
we consider a sequence of random graphs $(G^{(n),p})_{n \ge 1 }$ starting from the unique graph with one vertex and defined recursively.
Given $G^{(n),p}$, a graph with $ n $ vertices, we obtain $G^{(n+1),p}$,
 a graph with $ n + 1 $ vertices, by the following procedure.
\begin{enumerate}
\item Let $v$ be a uniform random vertex of $G^{(n),p}$.
\item Add a new vertex $v'$ to $G^{(n),p}$, with the same set of neighbours as $v$.
\item With probability $1-p$, we connect $v$ and $v'$ with an edge.
\end{enumerate}

\begin{figure}
\[\includegraphics[height=3cm]{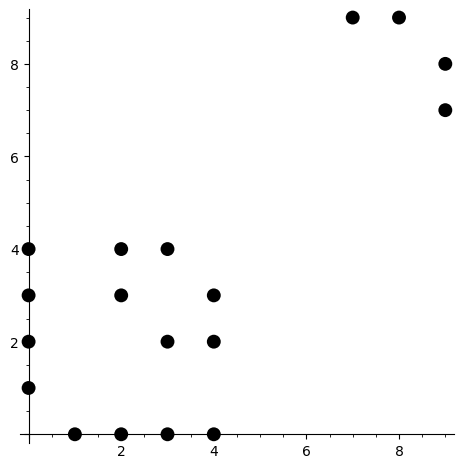}
\qquad  \includegraphics[height=3cm]{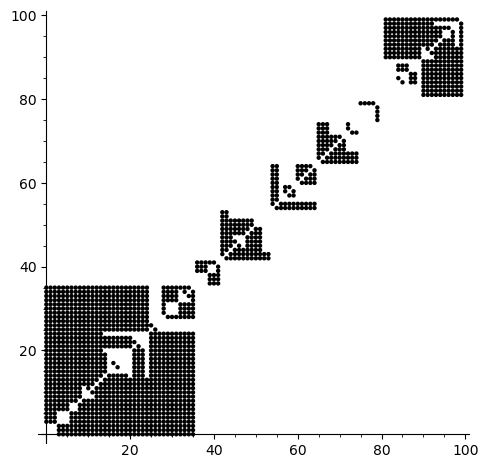}
\qquad  \includegraphics[height=3cm]{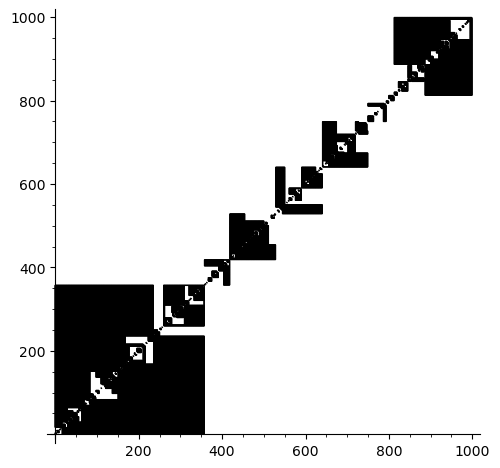}\]
\caption{A sample of graphs $G^{(10),{\, 0.5}}$, $G^{(100),{\, 0.5}}$,
$G^{(1000),{\, 0.5}}$ corresponding to the same realization of the process 
$(G^{(n),{\, 0.5}})_{n \ge 1}$.}
\label{fig:recursive_cographs}
\end{figure}
A simulation of this random graph process for $p=1/2$ is shown on Figure~\ref{fig:recursive_cographs}. Here graphs are represented by a collection of dots,
namely two dots at coordinates $(i,j)$ and $(j,i)$ for each edge $\{i,j\}$ in the graph
(in some sense, this is a pictorial version of the adjacency matrix of the graph).
We now state a convergence result for $G^{(n),p}$, which is the analogue of
Theorem~\ref{thm:convergence}. We assume the reader to be familiar with the notion of graphon
convergence.
\begin{theo}\label{thm:convergence_graphs}
The random graphs $ ( G^{(n),p} )_{ n \ge 1 } $ converge a.s.~to 
a random graphon,
which we call {\em recursive cographon (of parameter $p$)} and denote by $\Wrec_p$.
\end{theo}
A representative of the limiting graphon $\Wrec_p$ can be constructed using the random order~$\prec$ on $[0,1]$,
which we define later in Section~\ref{sec convergence}. Namely, for $x<y$ in $[0,1]$
we set
\[\Wrec_p(x,y)=\Wrec_p(y,x) = \begin{cases} 0 &\text{ if }x \prec y ;\\
1 &\text{ if }y \prec x .\end{cases}\]

Moreover, $\Wrec_p$ has the following properties, which are analogues
of Propositions~\ref{prop:singularity}, \ref{prop:self_similarity} and \ref{prop:expected_densities_murec}.
\begin{prop}\label{prop:singularity_graphs}
Let $p \ne q$ be fixed in $(0,1)$.
Then the distributions of the four random graphons $\Wrec_p$, $\Wrec_q$, $\WBr_p$
and $\WBr_q$ are pairwise singular,
where $\WBr_p$ is the Brownian cographon of parameter $p$ introduced in \cite{bassino2022cographs,stufler2021cographs}.
\end{prop}

\begin{prop}
\label{prop:self_similarity_graphs}
Fix $p$ in $[0,1]$, and let $W_1$ and $W_2$ be independent copies of $\Wrec_p$.
Let also $U$ be a uniform random variable in $[0,1]$ and $S$ be a Bernoulli random variable of parameter $1-p$, independent from each other and from $(W_1,W_2)$.
We define a graphon $W$ by
\[ W(x,y) = \begin{cases}
W_1(x/U,y/U) &\text{ if }x,y \le U;\\
W_2((x-U)/(1-U),(y-U)/(1-U)) &\text{ if }x,y > U;\\
S &\text{ if }x  \le U <y\text{ or }y  \le U <x. 
\end{cases}\]
Then $W$ has the same law as $\Wrec_p$.
Moreover the law of $\Wrec_p$ is characterized by this property.
\end{prop}

For the last statement, we write $\dens(H,W)$ for the (induced) density of $H$
in a graphon~$W$.
We also recall that cographs can be encoded by decorated trees called
cotrees, see, e.g., \cite[Section 2.2]{bassino2022cographs}.
Here, we will consider cotrees where internal nodes are labeled
 by integers from $1$
to some $k$, 
and say that a cotree is increasing if labels are increasing from the root to the leaves. 
\begin{prop}
\label{prop:expected_densities_Wrec}
For any graph $H$ of size $n$,
we have
\[\mbb E\big[\dens(H,\Wrec_p) \big] = \mbb P\big[G^{(n),p} =H \big] = \frac{N_{inc}(H)}{(n-1)!}\, p^{Z(H)}\, (1-p)^{n-1-Z(H)},  \]
where $N_{inc}(H)$ the number of increasing binary cotrees $T$ encoding $H$,
and $Z(H)$ is the number of decorations $0$ in any binary cotree encoding $H$.
\end{prop}

All of these results are easily obtained, either by applying the inversion graph mapping
to the permutation results, or by adapting the proofs.
Since the space of graphons has no natural convex structure,
there is no natural notion of expectation of the random graphon $\Wrec$,
and Proposition~\ref{prop:intensity_beta} and Corollary~\ref{corol:intensity_density}
have no analogues for $\Wrec$.

\subsection{Outline of the paper} The remainder of this paper is organized as follows.
In Section~\ref{ssec:preliminaries}, we discuss some background material that is needed later. 
In Section~\ref{sec convergence}, we go through an explicit construction of the recursive separable permuton and prove Theorem~\ref{thm:convergence}.
In Section~\ref{sec:properties}, we investigate the properties of the recursive separable permuton and prove all of our other results.

\section{Background}
\label{ssec:preliminaries}

This section gathers some material needed in the rest of the paper.
The first two items (permutation patterns and the Wasserstein metric) consist of standard material. 
The last item (push-forward permutons) is more specific to this project.

\subsection{Permutation patterns.}
If $\sigma$ is a permutation of size $n$ (we simply write \enquote{permutation of $n$}
from now on) and $I$ a subset of $\{1,\dots,n\}$ with $k$ elements,
then the {\em pattern} induced by $\sigma$ on the set of positions $I$ is the unique permutation
$\tau=\pat_I(\sigma)$ of $k$ with the following property: writing $I=\{i_1,\dots,i_k\}$ with $i_1<\dots<i_k$,
we have, for all $1 \le g,h \le k$,
\[ \sigma(i_g) < \sigma(i_h) \ \Leftrightarrow \ \tau(g) < \tau(h).\]
In other words, $ \tau $ is obtained by considering the subsequence $\sigma(i_1)\,\sigma(i_2)\,\ldots\, \sigma(i_k)$
of $\sigma$ and replacing the smallest element by $1$, the second smallest by $2$, and so on.
For example, the pattern induced by $3\, 2\, 5\, 6\, 4\, 7 \,1$ on positions $\{2,3,5\}$
is $132$.

\subsection{Wasserstein metric.}
Some arguments in Sections~\ref{sec:SelfSimilarity} and \ref{sec:Intensity} use the notion
of {the} Wasserstein distance between probability distributions.
We briefly recall some main facts about it.

Let $X$ be a complete metric space with distance $d_X$. 
For each $p\ge 1$, consider the space $\mathcal M^p_1(X)$ of probability measures on $X$
with a finite $p$-th moment.
This space can be endowed with the so-called {\em $p$-th Wasserstein distance} 
(also called {the} optimal cost distance or {the} Kantorovich-Rubinstein distance):
\[ 
	d_{W,p}(\nu,\nu')^p
		\ldef 
			\inf_{ 
					\substack{ \bm X,\bm X': \\ \bm X \sim \nu, \, {\bm X'} \sim \nu' }
				} 
    			\mathbb E[ d_X(\bm X,{\bm X'})^p],
\]
where the infimum is taken over all pairs $(\bm X, {\bm X'})$ of random variables defined on the same probability space with distributions $\nu$ and $\nu'$, respectively.
{It is well-known that this defines} a metric on $\mathcal M^p_1(X)$;
see, e.g., \cite{clement2008wassertein}.

In this article, we will be interested in the cases $X=[0,1]$ and $X=[0,1]^2$.
In these cases, and more generally whenever $X$ is compact,
we have $\mathcal M^p_1(X)=\mathcal M_1(X)$ for all $p$.
Also, convergence {in} the Wasserstein metric is equivalent to the weak convergence of measures
(in general, on non necessarily compact spaces, it is equivalent to weak convergence
and convergence of the $p$-th moment).
In particular, it is sometimes fruitful to think of permuton convergence as convergence
for the (first) Wasserstein distance.

\subsection{Push-forward permutons.}
\label{ssec:push_forward}
Given a function $f:[0,1] \to [0,1]$, we can consider the measure $\mu_f:=( \, \cdot \,, f( \, \cdot \, ) )_{ \# }  \Leb $ on $[0,1]^2$, 
i.e.~the push-forward of the Lebesgue measure on $[0,1]$ by the map $x\mapsto (x,f(x))$.
The projection of $\mu_f$ on the first coordinate is always the Lebesgue measure on $[0,1]$,
while its projection on the second coordinate is $f_{ \# }  \Leb$.
Thus, if $f$ preserves the Lebesgue measure, then $\mu_f$ is a permuton.
Such permutons will be referred here as {\em push-forward} permutons.
Both the recursive and the Brownian separable permutons are push-forward permutons (for random functions $f$).

The permuton $\mu_\pi$ associated to a permutation is not a push-forward
permuton. We can, however, approximate it by a push-forward measure easily.
Given a permutation $\pi$ of $n$, 
let us define the function $f_{ \pi } 
		\colon
			[ 0, 1 ] \to [ 0, 1 ]
$
 by $f_\pi(0)=0$ and, for $x>0$,
$$
	f_{ \pi }( x )
		=
			\frac 	{
						\pi( \ceiling{ n x } )
					}{
						n
					}
		.
$$
We also write $ \hat{\mu}_{ \pi}=\mu_{f_\pi}
		=
			( \, \cdot \,, f_{ \pi }( \, \cdot \, ) )_{ \# }
			\Leb.
$
The measure $\hat{\mu}_{ \pi}$ is not a permuton (its projection on the $y$-axis is not uniform), but it resembles the permuton $\mu_{ \pi}$:
while $\mu_{ \pi}$ has, for each $i$, a mass $1/n$ uniformly distributed on the square
$[\frac{i-1}{n};\frac in] \times [\frac{\pi(i-1)}{n};\frac{\pi(i)}{n}]$, the measure  $\hat{\mu}_{ \pi}$
has the same mass distributed on the segment $[\frac{i-1}{n};\frac in] \times \big\{\frac{\pi(i)}{n}\big\}$.
This clearly implies
\begin{equation}
	\label{remark describing permutations as pushforwards}
	d_{ W, 1 }
    	\left(
        	\hat{\mu}_{ \pi },
        	\mu_{ \pi }
    	\right)
    		\le
    			\frac{ 1 }{ n }
\end{equation}
for any permutation $\pi$ of $n$.
It follows that, given a sequence of permutations $\pi^{ ( n ) } $ of increasing size,
 the sequences
$
	\{
		\hat{\mu}_{ \pi^{ ( n ) } }
	\}_{ n \ge 1 }
$
and
$
	\{
		\mu_{ \pi^{ ( n ) } }
	\}_{ n \ge 1 }
$
have the same limit points in $ \mc M_1( [ 0, 1 ]^2 ) $.

We end this section with a convergence criterium
for push-forward measures $\mu_f$,
which will be  used in the proof of our main result.
Since particular cases of push-forward permutons have been studied in the literature
\cite{bassino2018BrownianSeparable,borga2021skew,borga2022meanders},
it might also be useful in other contexts.
We also refer the reader to a paper of 
Bhattacharya and Mukherjee \cite{bhattacharya2017degree},
for a related result connecting pointwise convergence of random permutations seen
as functions,
and convergence of the associated random permutons.
\begin{prop}
	\label{prop characterization pushforward convergence}
	Let
	$ f, f_1, f_2, \ldots $ be measurable functions on $[ 0, 1 ]$ with values in $[0,1]$.
	Then the following statements are equivalent:
	\begin{enumerate}[ label = (\roman*) ]
		\item
		\label{claim some lp convergence}
		$
			f_n
				\to
					f
		$
		in $ L^p $ for some $ p \in [ 1, \infty ) $,

		\item
		\label{claim all lp convergence}
		$
			f_n
				\to
					f
		$
		in $ L^p $ for all $ p \in [ 1, \infty ) $,

		\item
		\label{claim weak convergence of pushforwards}
		$
			( \, \cdot \,, f_n( \, \cdot \, ) )_{ \# } 
			\Leb
				\to
        			( \, \cdot \,, f( \, \cdot \, ) )_{ \# } 
        			\Leb
		$
		weakly.

	\end{enumerate}
\end{prop}

\begin{proof}
{Letting $U$ be a uniform random variable in $[0,1]$ on some probability space $\Omega$,
the convergence of $f_n$ to $f$ in $L^p([0,1])$ is equivalent to
the convergence of the random variables $f_n(U)$ to $f(U)$ in $L^p(\Omega)$.
Since all $f_n(U)$ and $f(U)$ are uniformly bounded,
(i) is equivalent to (ii), and they are both equivalent the weak convergence of $f_n(U)$ to $f(U)$.}

{It remains to show the equivalence with (iii). If (iii) holds, the pair $(U,f_n(U))$ converges weakly to $(U,f(U))$.
Restricting to the second coordinates, we know that $f_n(U)$
converges weakly to $f(U)$, implying (i).}

{Conversely, if (i) holds, $f_n(U)$ converges weakly to $f(U)$,
and thus $(U,f_n(U))$ converges weakly to $(U,f(U))$,
{\em i.e.}~(iii) holds. This completes the proof of the proposition.}
\end{proof}

\subsection{{Permutations and pairs of total orders}}
\label{ssec:permutation_pair_orders}
{In some constructions, it will be convenient to see a permutation
as a finite set endowed with two total orders.
This is a standard point of view in the context of permutation patterns,
see, {\em e.g.}, \cite{albert2020logic}.}

Let us explain how to associate a permutation with a triple
$(E,<,\prec)$, where $ E $ is a finite set equipped with two total orders $<$ and $\prec$.
Indeed, using the first order, we write $E=\{x_1,\dots,x_k\}$ where $k=|E|$ and $x_1<\dots<x_k$.
Then there exists a unique permutation $ \sigma $ satisfying
\begin{equation}
	\label{defn permutation from a set}
	\sigma( j ) < \sigma( k )
		\quad
		\Longleftrightarrow
		\quad
			x_j \prec x_k.
\end{equation}

We denote this permutation by $\sigma= \Perm( E,<,\prec ) $.
For example, we choose \hbox{$E=\{a,b,c,d,e\}$}, with $a<b<c<d<e$ and $d \prec b\prec e\prec c \prec a$,
then $\sigma=52413$. To visualize this construction,
we can represent elements of $E$ as points in the plane, 
so that $<$ compares the $x$-coordinates, while $\prec$ compares the $y$-coordinates.
Then $E$ ressembles the diagram of $\sigma$, see Figure~\ref{fig:example_perm}.
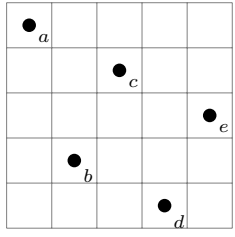
\begin{figure}
\[\begin{tikzpicture}[scale=0.6]
  \draw[gray,very thin] (0,0) grid (5,5);
  \fill (0.5,4.5) circle (0.15) node[below right =-0.1] {\scriptsize $a$};
  \fill (1.5,1.5) circle (0.15) node[below right =-0.1] {\scriptsize $b$};
  \fill (2.5,3.5) circle (0.15) node[below right =-0.1] {\scriptsize $c$};
  \fill (3.5,0.5) circle (0.15) node[below right =-0.1] {\scriptsize $d$};
  \fill (4.5,2.5) circle (0.15) node[below right =-0.1] {\scriptsize $e$};
\end{tikzpicture}\]
\caption{A set of points $E=\{a,b,c,d,e\}$; comparing $x$-coordinates and
$y$-coordinates yield two orders $<$ and $\prec$ on $E$.
The associated permutation $\sigma= \Perm( E,<,\prec )$ is $52413$.}
\label{fig:example_perm}
\end{figure}

Taking patterns is simple with this viewpoint.
If $E=\{x_1,\dots,x_k\}$ is as above (in particular, assuming $x_1<\dots<x_k$)
and if $I$ is a subset of $\{1,...,k\}$,
then 
\[\pat_I \big(  \Perm( E,<,\prec ) \big) =  \Perm \big(\{x_i,i \in I\},<,\prec\big). \]

Finally, $\sigma=\Perm( E,<,\prec )$ is also given by an explicit formula.
As above, let $x_k$ be the $k$-th smallest
element in $E$ for the first order $<$.
Then, if $x_k$ is the $\ell$-th smallest element for the order $\prec$, we have
\begin{equation}
  \sigma(k) = \ell = 1+\sum_{x \in E} \bm 1[x \prec x_k].
  \label{eq formula sigma}
\end{equation}

\section{Construction and convergence}
\label{sec convergence}

In this section, we construct the {recursive separable permutations} and the recursive separable permuton on a common probability space and establish the almost sure convergence of the permutations to the permuton. 

Fix $p$ in $(0,1)$.
Throughout the section, we consider two independent random i.i.d.~sequences $(U_j)_{j \ge 1}$ and $(S_j)_{j \ge 1}$, where the $U_j$ are uniform in $[0,1]$
and the $S_j$ are random signs  in $ \{ \oplus, \ominus \} $ with $\mathbb P( S_j = \oplus ) = p $.
For convenience, the signs $ \oplus $ and $ \ominus $ will often be regarded as $ 1 $ and $ -1 $, respectively, and we set $ U_0 = 0 $ and $ U_{ - 1 } = 1 $.

To illustrate our construction, we sampled sequences $(U_j)_{j \ge 1}$ and $(S_j)_{j \ge 1}$ as above (setting $p=1/2$),
and used them in examples throughout the section. The first sampled values (rounded to 2 decimal digits for $U_j$) are
\begin{align} U_1=\ &0.72,\ U_2=0.82,\ U_3=0.54,\ U_4=0.13,\ U_5=0.60, \nonumber\\
  &S_1=\oplus,\ S_2=\ominus,\ S_3=\ominus,\ S_4=\ominus, S_5=\oplus.\label{Eq:Values_US}
\end{align}

\subsection{The general strategy}
\label{ssec:strategy}
To construct our permutations and permuton, we will construct a sequence of partial orders on [0, 1] that will give rise to a total order on [0, 1] in \enquote{the limit}.
These partial orders will then be identified as permutations 
(in a similar spirit as in Section~\ref{ssec:permutation_pair_orders}),
while the total order will lead to a permuton. 

With $ ( U_j )_{ j \ge 1 } $ and $ ( S_j )_{ j \ge 1 } $ defined as above, we define our first partial order by using $ U_1 $ to split $ [ 0, 1 ] $ into the two intervals $ [ 0, U_1 ) $ and $ [ U_1, 1 ] $ and declaring that every element in $ [ 0, U_1 ) $ is less than every element in $ [ U_1, 1 ] $ if $ S_1 = \oplus $ and that $ [ U_1, 1 ] $ is instead the smaller interval if $ S_1 = \ominus $. The second partial order is to be obtained from the first one by using $ U_2 $ to further split one of the intervals into two new intervals and declare that the leftmost interval is the smaller of the two if $ S_2 = \oplus $ and that it is the larger of the two if instead $ S_2 = \ominus $. We repeat this process to obtain a sequence of partial orders, at each step using the next uniform random variable to split an interval into two intervals and declaring the leftmost interval to be the smaller interval when the corresponding sign is $ \oplus $ and the larger interval otherwise. 
See Figure~\ref{fig:construction_order}.
\begin{figure}
  \begin{center}
    \includegraphics[width=\textwidth]{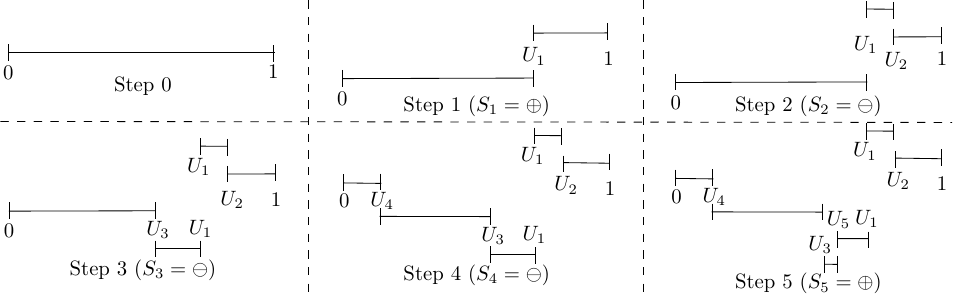}
  \end{center}
  \caption{Illustration of the first five partial orders associated with the sequences $(U_j)_{j \le 5}$ and $(S_j)_{j \le 5}$
  given in \eqref{Eq:Values_US}. An interval drawn at a higher level than another one means that its elements are larger.}
  \label{fig:construction_order}
\end{figure}
\smallskip

We show that these partial orders can be identified as the recursive separable permutations in Section \ref{ssec:construction_permutations}, but this can be seen intuitively. Indeed, the splitting of intervals into two and then assigning a sign to compare them mimics the inflation operation. Moreover, the distribution of the signs ensures that increasing adjacencies are created at the right rate. Finally, the properties of the uniform sequence ensure that the interval to be split at each step is chosen uniformly at random, as in our permutation model. Another motivation for the choice of the uniform sequence is the fact that the proportion of intervals to the left of $ U_1 $ must resemble the proportion of points in $ \sigma^{ ( n ), p } $ that were obtained from subdividing the leftmost point in $ \sigma^{ ( 2 ), p } $.
The number of such points has the same law as the number of white balls in the standard P\'olya urn process,
and its proportion is known to converge a.s.~to a uniform random variable in $[0,1]$.
\smallskip

The above partial orders naturally lead in the limit $n \to \infty$
to a total order on $[ 0, 1 ]$. %
 Thus we expect that the limiting permuton can be retrieved from this order. This retrieval is executed in Section~\ref{sec:construction_permuton},
as a natural extension of the construction of \cref{ssec:permutation_pair_orders}. 

\subsection{The total order and the limiting permuton}
\label{sec:construction_permuton} 
We start by offering an explicit formal description of the order described informally in the previous section.
Given $ x < y $ in $ [ 0, 1 ] $ such that 
$ 
	( x, y ] 
  		\cap
        	\{ U_j \}_{ j \ge 1 } 
        		\neq 
        			\emptyset 
        		,
$
we consider the minimal $ j \ge 1 $ such that 
$ 
	U_j
		\in 
			(x,y] 
		.
$
We set
\begin{equation}
 \begin{cases}
x \prec y &\text{ if } S_j = \oplus;\\
y \prec x &\text{ if } S_j = \ominus.
\end{cases}
\end{equation}
Alternatively, letting
\begin{align*}
	\mc I_{ x, y }
		& =
			\Big( \!
				\min( x, y ), \, \max( x, y ) 
			\Big]
		,
			\qquad
			x, y \in [ 0, 1 ]
		,
		\\
	i_{ x, y }
		& =
			\inf 
			\left\{ 
					j \ge 1 
				: 
					U_j \in \mc I_{ x, y } 
			\right\}
		,
			\qquad
			x, y \in [ 0, 1 ]
		,
\end{align*}

\noindent
and using the convention
$$
	S_\infty
		=
			0
		,
$$

\noindent
we can describe this relation concisely as 
$$
	x \prec y
		\qquad
		\Longleftrightarrow
		\qquad
			( y - x )
			S_{ i_{ x, y } }
				>
					0
				.
$$

Using the random relation $\prec$, we define a random function $ \phi \colon [ 0, 1 ] \to [ 0, 1 ] $ and a random measure on $ [ 0, 1 ]^2 $ by
\begin{align}
	\phi(x) 
		& = 
			\Leb\left( \{y \in [ 0, 1 ] : y \prec x\} \right)
		, \label{eq:def_phi}
		\\
	\murec_p
		& = \mu_\phi =
			( \, \cdot \,, \phi( \, \cdot \, ) )_{ \# }
			\Leb, \label{eq:def_mu_phi}
\end{align}
where, in the last equation, we use the notation $\mu_f$ of Section~\ref{ssec:preliminaries}.
{We note that \eqref{eq:def_phi} is a natural infinite version 
of \eqref{eq formula sigma}, while \eqref{eq:def_mu_phi} mimicks the way we associate a measure $\hat\mu_\pi$ with a permutation $\pi$ 
(see Section~\ref{ssec:push_forward}).}

In Proposition \ref{properties of phi} below, we will see that the following properties hold almost surely:
the relation $ \prec $ is a total order on $ [ 0, 1 ] $,
the function $ \phi $ is Lebesgue-preserving, 
and 
the measure $ \murec_p $ is a permuton.
We define the {\em recursive separable permuton (of parameter $p$)} as the random measure $ \murec_p $ on $[0,1]^2$.

\begin{rema}
  \label{rem:Brownian}
{  Our construction is inspired by the construction of the (biased) Brownian separable permuton $\muBr_p$ given in \cite{maazoun2020BrownianPermuton}.
  Indeed, the construction of $\muBr_p$ also goes through the construction of a random order $\prec$.
  However, instead of taking i.i.d.~uniform random variables $\{U_j\}_{j \ge 1}$, we consider
  a Brownian excursion $\exc$ on $[0,1]$ and attach signs to its local minima.
  We then define, for $x<y$ in $[0,1]$,
  \[ \begin{cases}
x \prec y &\text{ if } S_m = \oplus;\\
y \prec x &\text{ if } S_m = \ominus,
\end{cases} \]
 where $S_m$ is {\em the sign attached to the local minimum $m$ of $\exc$ on the interval $[x,y]$}. The rest of the construction, namely Eqs.~\eqref{eq:def_phi} and \eqref{eq:def_mu_phi}, is the same in the constructions of $\muBr_p$ and $\murec_p$.
  As said in the introduction, even though the constructions are similar,
  the resulting random permutons are different, and even more, their distributions
  are singular to each other (see Section~\ref{ssec:Singular_Distributions}). }
\end{rema}

\subsection{The Glivenko-Cantelli theorem and some consequences}
{For $\murec_p$ to be a permuton,
we need the sequence $(U_j)_{j \ge 1}$ to satisfy some good properties.
The fact that these properties hold a.s.~is essentially given by the well-known 
Glivenko-Cantelli theorem, which we now recall,
together with straightforward consequences.
To this end,}
we consider the following objects, for any $n \ge 1$:
\begin{itemize}
\item the ordered statistics of $ \{ U_{ - 1 }, U_0, \ldots, U_n \} $,
$$
	0=U_{ ( 0, n ) }
    	\le
        	U_{ ( 1, n ) }
        		\le
                	\ldots
                		\le
                        	U_{ ( n, n ) }
							\le
								U_{ ( n + 1, n ) }=1;
$$
\item 
the length of the largest interval in $ [ 0, 1 ] \setminus \{ U_1, \ldots, U_n \} $, 
$$
	\Delta_n
		=
			\max_{ 1 \le j \le n + 1 }
				U_{ ( j, \, n ) }
			-
				U_{ ( j-1, \, n ) }
		,
$$

\noindent
and

\item
the empirical measures
$$ 
	{P_n}
		= 
			\frac{1}{n} 
			\sum_{ j = 0 }^{ n - 1 }
				\delta_{ U_j },
$$
where $\delta_x$ is the Dirac measure at $x$. 
\end{itemize}

\noindent
In addition, we consider the event
\begin{equation}
	\label{defn almost sure set}
	\mc E
		=
			\Bigg
			\{
				( U_j )_{ j \ge 1 }
				\text{ are distinct},
				\,
			\sup_{ x \in [ 0, 1 ] }
			 	\bigg|
					x 
				- 
					\frac{1}{n} 
					\sum_{ j = 1 }^n 
						\indicator( U_j \le x )
				\bigg|
						\xrightarrow[ n \to \infty ]{}
							0
			\Bigg
			\}.
\end{equation}

\noindent
The following result summarizes the properties of these objects that will be useful to us.

\begin{prop}
	\label{prop consequences Glivenko-Cantelli}
	
	The event $ \mc E $ occurs almost surely and on $ \mc E $, the following statements hold (all limits are to be taken as $ n \to \infty $):
	\begin{enumerate}[ label = (\roman*) ]
		\item
		\label{claim convergence of ordered statistics}
		$
    		\sup_{ t \in [ 0, 1 ] }
    		 	\left|
    				U_{ ( \ceiling{ n t }, \, n - 1 ) } 
    			- 
    				t
    			\right|
    				\longrightarrow
    					0
					,
		$

		\item
		\label{claim convergence of maximal gap}
		$
			\Delta_n
				\to
					0
				,
		$

		\item
		\label{claim convergence of empirical measures over intervals}
		for every $ k \ge 1 $,
		$
    		\sup_{ J \in \mc A_k }
    		 	\left|
    				\Leb( J )
    			- 
    				{P_n}( J )
    			\right|
    				\longrightarrow
    					0
					,
		$
    	where $ \mc A_k $ consists of all subsets of $ [ 0, 1 ] $ that are a disjoint union of at most $ k $ intervals.
	\end{enumerate}
\end{prop}

\begin{proof}
	The fact that $ \mc E $ occurs almost surely
	is a classical result known as the Glivenko-Cantelli Theorem.	
	We assume for the remainder of the proof that $ \mc E $ occurs.

	Since $ U_{ ( \ceiling{ n t }, \, n - 1 ) } \in [ 0, 1 ] $ for all $ t \in [ 0, 1 ] $ and $ n \ge 1 $, we can apply 
	the limit in (\ref{defn almost sure set}): as $ n \to \infty $, we have the convergence
	\begin{equation}\label{eq:U_close_FempU}
		\sup_{ t \in [ 0, 1 ] }
		 	\left|
				U_{ ( \ceiling{ n t }, \, n - 1 ) } 
			- 
				\frac{1}{n} 
				\sum_{ j = 1 }^n 
					\indicator( U_j \le U_{ ( \ceiling{ n t }, \, n - 1 ) } )
			\right|
				\longrightarrow
					0
				.
	\end{equation}
	
	\noindent
	Recalling that the sequence $ ( U_j )_{ j \ge 1 } $ contains distinct values, we obtain the estimate
	\begin{equation*}
    	\left|
			\frac{1}{n}
			\sum_{ j = 1 }^n
    			\indicator( U_j \le U_{ ( \ceiling{ n t }, \, n - 1 ) } )
		-
			t
		\right|
    				 \le 
                    		\frac{1}{n}
						+	
							\left|
								\frac{1}{n}
                    			\sum_{ j = 1 }^{ n - 1 }
                        			\indicator( U_j \le U_{ ( \ceiling{ n t }, \, n - 1 ) } )
                    		-
                    			t
                    		\right|
    			 \le 
                    		\frac{2}{n}
	\end{equation*}
	
	\noindent
	which shows that $ t $ can replace
	$
		\frac{1}{n} 
		\sum_{ j = 1 }^n 
			\indicator( U_j \le U_{ ( \ceiling{ n t }, \, n - 1 ) } )
	$
	in the convergence \eqref{eq:U_close_FempU}.
	This establishes \ref{claim convergence of ordered statistics}. 
	Writing
	$$
		\left|
			U_{ ( j, \, n ) }
		-
			U_{ ( j-1, \, n ) }
		\right|
			\le
        		\left|
        			U_{ ( j, \, n ) }
        		-
					\frac{ j }{ n }
        		\right|
				+
        		\left|
					\frac{ j }{ n }
				-
					\frac{ j - 1 }{ n }
        		\right|
				+
        		\left|
					\frac{ j - 1 }{ n }
				-	
        			U_{ ( j-1, \, n ) }
        		\right|
	$$
	
	\noindent
	for $ n \ge 1 $ and $ 1 \le j \le n + 1 $, we see that \ref{claim convergence of maximal gap} follows immediately from \ref{claim convergence of ordered statistics}.

	For $ x \in [ 0, 1 ] $, we can write
	\begin{align*}
		\left|
			\Leb( [ 0, x ] )
		- 
			{P_n}( [ 0, x ] )
		\right|
			& = 
            		\Bigg|
            			x
            		- 
						\frac{ 1 }{ n }
						\sum_{ j = 0 }^{ n - 1 }
							\indicator( U_j \le x )
            		\Bigg|
			\\
			& \le 
            		\Bigg|
            			x
            		- 
						\frac{ 1 }{ n }
						\sum_{ j = 1 }^n
							\indicator( U_j \le x )
            		\Bigg|
				+
					\frac{ 1 }{ n }
					\left|	
						\indicator( U_n \le x )
					-
						\indicator( U_0 \le x )
					\right|
	\end{align*}
	
	\noindent
	and apply the limit in (\ref{defn almost sure set}) to see that
	$$
		\sup_{ x \in [ 0, 1 ] }
    		\left|
    			\Leb( [ 0, x ] )
    		- 
    			{P_n}( [ 0, x ] )
    		\right|
				\longrightarrow
					0
	$$
	
	\noindent
	as $ n \to \infty $.
	Using the decomposition 
	$$
		\nu( ( x, y ] )
			=
				\nu( [ 0, y ] )
			-
				\nu( [ 0, x ] )
	$$
	
	\noindent
	for any probability measure $ \nu $ on $ [ 0, 1 ] $ and $ x < y $ in $ [ 0, 1 ] $, we then obtain the convergence	
	$$
		\sup_{ x < y }
    		\left|
    			\Leb( ( x, y ] )
    		- 
    			{P_n}( ( x, y ] )
    		\right|
				\longrightarrow
					0
	$$
	
	\noindent
	as $ n \to \infty $.
	Recalling that the sequence $ ( U_j )_{ j \ge 1 } $ contains distinct values, we have the bound
	\begin{equation*}
		P_n\{ x \}
			= 
					\frac{ 1 }{ n }
					\sum_{ j = 0 }^{ n - 1 }
						\indicator( U_j = x )
		 \le 
					\frac{ 1 }{ n }
	\end{equation*}

	\noindent
	for any $ x \in [ 0, 1 ] $. 
	In particular, we can add or remove the endpoints of the interval $ ( x, y ] $ above while maintaining the convergence.
	This establishes the result for $ k = 1 $.
	The extension to general $ k $ follows from the additivity property of measures.
\end{proof}

\subsection{Useful properties of the relation $\prec$ and the function $\phi$}

This section is dedicated to analyzing the relation $ \prec $ and the function $ \phi $. The goal is twofold:
not only will this allow us to establish that $\murec_p $ is a permuton a.s., but it will prepare us for the convergence argument.
We begin with basic properties of the relation $ \prec $.

\begin{lemm}
	
	\label{lem extrema and the random order}
	Let $ x, y, z \in [ 0, 1 ] $.
	Then
	$
		i_{ x, y },
	$
	$
		i_{ y, z },
	$
	and
	$
		i_{ x, z }
	$
	can be assigned the labels $ a $, $ b $, and $ c $ so that
	$$
		a = b 
			\le 
				c
			.
	$$

	\noindent
	Moreover, if
	$$
		i_{ x, y }
			=
				i_{ y, z }
			<
				\infty
			,
	$$
	
	\noindent
	then
	$$
		y 
			\prec
				x, z
				\qquad
				\text{or}
				\qquad
			x, z
				\prec
					y
				.
	$$

\end{lemm}

\begin{proof}
	The intervals $ \mc I_{ x, y } $, $ \mc I_{ y, z } $, and $ \mc I_{ x, z } $ can be assigned the labels $ A $, $ B $, and $ C $ so that they satisfy
	\begin{equation}
		\label{eqn interval identity}
		A
			=
				B \sqcup C
			.
	\end{equation}
	
	\noindent
	It follows immediately that
	\begin{align*}
		\inf \left\{ j \ge 1 : U_j \in A \right\}
			& = 
				\inf \left\{ j \ge 1 : U_j \in B \sqcup C \right\}
			\\
			& = 
				\min \!
				\big(
    				\inf \left\{ j \ge 1 : U_j \in B \right\} \!,
					\,
    				\inf \left\{ j \ge 1 : U_j \in C \right\}
					\!
				\,
				\big)
			,
	\end{align*}
	
	\noindent
	establishing the first claim. 
	Assume now that
	$
		i_{ x, y }
			=
				i_{ y, z }
			<
				\infty
			,
	$
	or equivalently,
	$$
		\inf \left\{ j \ge 1 : U_j \in \mc I_{ x, y } \right\}
			=
				\inf \left\{ j \ge 1 : U_j \in \mc I_{ y, z } \right\}
			<
				\infty
			.
	$$
	\noindent
	Then the intervals $ \mc I_{ x, y } $ and $ \mc I_{ y, z } $ are not disjoint
	$
    	(
			U_{ i_{ x, y } } 
    			\in
    				\mc I_{ x, y }
    				\cap
    				\mc I_{ y, z }
    	)
		.
	$
	This implies that $ y $ is either the minimum or maximum of $ \{ x, y, z \} $.
	In particular, the inequality $ ( y - z )/( y - x ) > 0 $ holds.
	From this, we obtain the equivalence
	\[
		x \prec y
			\  \Longleftrightarrow \
				( y - x )
				S_{ i_{ x, y } }
					>
						0
			\  \Longleftrightarrow\
				( y - z )
				S_{ i_{ y, z } }
					>
						0
			\ \Longleftrightarrow \
				z \prec y.
\]

	\noindent
	The equivalence
	$$
		y \prec x
			\quad \Longleftrightarrow \quad
				y \prec z
	$$
	
	\noindent
	can be obtained similarly. 
	Recalling that $ y \neq x $ and $ i_{ x, y } < \infty $, it follows that $ ( y - x ) S_{ i_{ x, y } } \neq 0 $, so either $ x \prec y $ or $ y \prec x $.
	Applying one of the above equivalences concludes the proof.
\end{proof}

\begin{coro}
	\label{cor partial order}
	The relation $ \prec $ is a partial order.
\end{coro}

\begin{proof}
Irreflexivity and asymmetry are straightforward to prove, let us consider transitivity.
	Suppose that $ x \prec y $ and $ y \prec z $.
	Since 
	$
		y 
			\not\prec
				x
			,	
	$
	$
		z
			\not\prec
				y
			,
	$
	and
	$
		i_{ x, y }
			,
			\,
				i_{ y, z }
			<
				\infty
			,
	$
	Lemma~\ref{lem extrema and the random order} implies that
	$
		i_{ x, y }
			\neq
				i_{ y, z }
			.
	$
	Applying Lemma~\ref{lem extrema and the random order} again reveals that
	$
		i_{ x, z }
			=
				\min(
					i_{ x, y },
					i_{ y, z }
					)
	$
	and $ x \prec z $.
	This shows that $ \prec $ is transitive.
\end{proof}

Now we turn our attention to the function $ \phi $.
For this analysis, we make use of the functions
$$
	\phi_k ( x )
		=
			\Leb\big( 
					\big\{ 
						y \in [ 0, 1 ]
					: 
						y \prec x, 
						\,
						i_{ x, y } 
							\le
								k
					\big\} 
				\big)
		,
			\qquad
			k \ge 0
		,
$$

\noindent
which serve as approximations to $ \phi $.
Indeed, the continuity of measure property implies that $ \phi_k \nearrow \phi $ pointwise as $ k \to \infty $.
In some sense, the function $\phi_k$ encodes the $k$-th partial order defined in \cref{ssec:strategy}.
See Figure~\ref{fig:graph_phik} for an illustration.
The basic properties of these functions are summarized in the following result.

\begin{figure}
  \begin{center}
    \includegraphics[width=.3\textwidth]{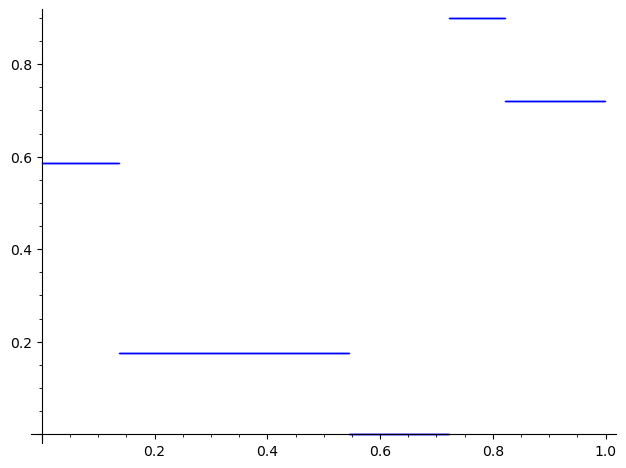} \quad \includegraphics[width=.3\textwidth]{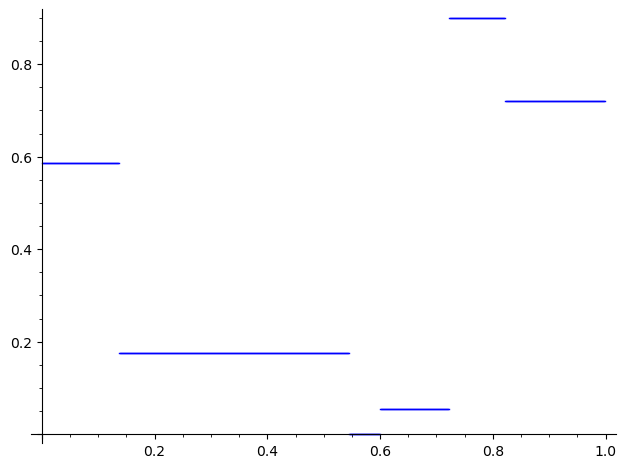} \quad \includegraphics[width=.3\textwidth]{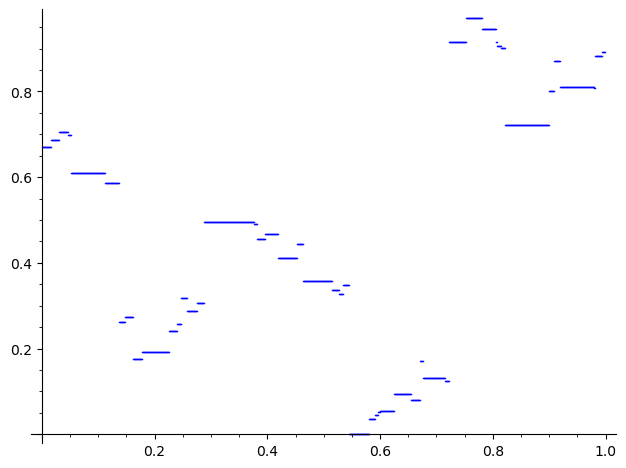}
  \end{center}
  \caption{Graphs of the functions $\phi_k$ for $k \in \{4,5,49\}$, associated with our sampled sequences $(U_j)_{j \le 5}$
  and $(S_j)_{j \le 6}$, whose first values are given in Eq.~\eqref{Eq:Values_US}.}
  \label{fig:graph_phik}
\end{figure}

\begin{prop}

	\label{prop characterization of discrete order}

	Let $ x, y \in [ 0, 1 ] $ and $ k \ge 0 $.
	The following statements hold:
	\begin{enumerate}[ label = (\roman*) ]
	
    	\item
        \label{claim equality in step function}
        $
        	i_{ x, y } > k
        		\qquad
    			\Longleftrightarrow
    			\qquad
            			\phi_k ( x ) 
    						= 
    							\phi_k( y )
    			,
        $

		\item
        \label{claim inequality in step function}
		$
        	x \prec y,
        	\,
        	i_{ x, y } \le k
        		\qquad
    			\Longleftrightarrow
    			\qquad
            			\phi_k( x ) < \phi_k( y ) 
    			,
        $
         \item 
        \label{claim alternative definition of phi_k}
    $   
    	\phi_k ( x )
    		=
    			\Leb\big( 
    					\big\{ 
    						y \in [0,1] 
    					: 
    						\phi_k( y ) < \phi_k( x ) 
    					\big\} 
    				\big)
    		;$
        \item
        \label{claim intervals in step function}
		$ \phi_k $ is constant on each of the intervals
        $$
        	\left[ 0, U_{ ( 1, k ) } \right),
        	\left[ U_{ ( 1, k ) }, U_{ ( 2, k ) } \right),            	
        	\ldots,
        	\left[ U_{ ( k - 1, k ) }, U_{ ( k, k ) } \right),
        	\left[ U_{ ( k, k ) }, 1 \right],
        $$
		\noindent
		and assumes different values on distinct intervals.
	\end{enumerate}

\end{prop}

\begin{proof}	
	\ref{claim equality in step function}
	$(\Longrightarrow)$
	Let $ x, y \in [ 0, 1 ] $ and $ k \ge 0 $.
	Suppose that $ i_{ x, y } > k $.
	Let us show that
	$
		L_x
			=
				L_y
			,
	$
	where
	$$
		L_t
			=
        		\{
        			s 
        		:
        			s \prec t,
					\,
        			i_{ s, t } \le k
        		\}
			,
				\qquad
				t \in [ 0, 1 ] 
			.
	$$
	
	\noindent
	Let $ z \in L_x $.
	This implies $i_{ x, z }  \le k$, and in particular
	$
		i_{ x, z } 
			<
				i_{ x, y }
	$.
	 From this, Lemma~\ref{lem extrema and the random order} gives us that
	$
		i_{ y, z } = i_{ x, z } 
			\le k
	$
	and
	$
		z \prec y.
	$
	Therefore, $ z \in L_y $, establishing the containment
	$
		L_x
			\subset
				L_y
			.
	$
	The reverse containment holds by symmetry.
	Writing
	$
		\phi_k( x ) 
			=
				\Leb( L_x )
			=
				\Leb( L_y )
			=
				\phi_k( y )
	$
	establishes the desired result.

	\ref{claim equality in step function} $(\Longleftarrow)$
	Suppose now that $ i_{ x, y } \le k $.
	It follows that $ x \neq y $ and $ S_{ i_{ x, y } } \neq 0 $, so either $ x \prec y $ or $ y \prec x $.
	Without loss of generality, we can assume that $ x \prec y $.
	Let us show that $ L_x \subset L_y $.
	Take $ z \in L_x $. 
	Since
	$
		z \prec x \prec y
	$,
		Corollary~\ref{cor partial order} implies that $ z \prec y $ and Lemma~\ref{lem extrema and the random order} implies that
	$
		i_{ x, z }
			\neq
				i_{ x, y }
			.
	$
	Applying Lemma~\ref{lem extrema and the random order} again, we find that
	$
		i_{ y, z } 
			= 
				\min( i_{ x, z }, i_{ x, y } )
			\le 
				k
			,
	$
	and consequently, $ z \in L_y $. 
	This establishes the containment $ L_x \subset L_y $.

	Thus we can write
	$
		\phi_k( y )
			=
				\phi_k( x ) 
			+ 
				\Leb( L_y \setminus L_x )
		,
	$
	and it only remains to show that $ L_y \setminus L_x $ has positive Lebesgue measure.
	To this end, we define
	$$
		U^-
			=
				\begin{cases}
    				\max\{
    						U_j
    					:
    						-1 \le j \le k,
						 	\,
    						 U_j \le x
    					\},
								&
									x \neq 1, 
					\\
    				\max\{
    						U_j
    					:
    						0 \le j \le k
    					\}
								&
									x = 1,
				\end{cases}
	$$
	
	\noindent
	and
	$$
		U^+
			=
				\begin{cases}
    				\min\{
    						U_j
    					:
    						-1 \le j \le k,
						 	\,
    						x < U_j
    					\},
								&
									x \neq 1, 
					\\
    			1
								&
									x = 1,
				\end{cases}
	$$
	
	\noindent
	and will show that $ ( U^-, U^+ ) $ is a nonempty interval contained in $ L_y \setminus L_x $.
	It should be clear that $ ( U^-, U^+ ) $ is nonempty,
	that $x \in [U^-,U^+)$ (except when $x=1$) and that
	$$
		( U^-, U^+ )
			\cap
				\{
					U_1, \ldots, U_k
				\}
					=
						\emptyset
					.
	$$
	
	\noindent
	Now take $ z \in ( U^-, U^+ ) $. We have $\mc I_{ x, z } \subset (U^-,U^+)$
	(except in the case $x=1$, where $U^+$ should be included).
	Recalling that $ U_j \in ( 0, 1 ) $ for $ j \ge 1 $, it can be verified that
	$$
		\mc I_{ x, z }
			\cap
				\{
					U_1, \ldots, U_k
				\}
			\subset
				( U^-, U^+ )
        			\cap
        				\{
        					U_1, \ldots, U_k
        				\} = \emptyset
			.
	$$
	
	\noindent
	It follows immediately that $ i_{ x, z } > k $, so $ z \notin L_x $.
	Since $i_{ x, y } \le k$ and $x \prec y$, we have $i_{ x, y }  < i_{ x, z } $
	and Lemma~\ref{lem extrema and the random order} gives us that
	$
		i_{ y, z }
			=
				i_{ x, y }
			\le 
				k
	$
	and $ z \prec y $, so $ z \in L_y $.
	This establishes the containment 
	$ 
		( U^-, U^+ ) 
			\subset 
				L_y \setminus L_x 
			,
	$ 
	concluding the proof.

	\ref{claim inequality in step function}
	$ ( \Longrightarrow ) $
	This statement was proved in the proof of \ref{claim equality in step function} $ ( \Longleftarrow ) $.

	\ref{claim inequality in step function}
	$ ( \Longleftarrow ) $
	Let $ x, y \in [ 0, 1 ] $ and $ k \ge 0 $ and suppose that 
	$
		\phi_k( x ) < \phi_k( y ).
	$
	Using \ref{claim equality in step function}, we have that
	$
		i_{ x, y }
			\le
				k
			.
	$
	As before, we must have that $ x \prec y $ or $ y \prec x $.
	However, the forward implication in \ref{claim inequality in step function} implies that
	$
		y \not\prec x.
	$
	Therefore, $ x \prec y $.

        \ref{claim alternative definition of phi_k}
        This claim follows from the definition of $\phi_k$ and \ref{claim inequality in step function}

	\ref{claim intervals in step function}
	This claim follows immediately from \ref{claim equality in step function}.
\end{proof}

The above result leads to the following estimate for $ \phi $, which plays a crucial role in the convergence argument.

\begin{coro}
	
	\label{cor discretization estimate}
	 Let $ k \ge 1 $ and $ \nu $ be a probability measure on $ [ 0 , 1 ] $.
	 The following inequality holds:
	\begin{enumerate}[ label = (\roman*) ]
    	\item
	 	\label{claim discretization estimate}
		$
    		\displaystyle
			\sup_{ x \in [ 0, 1 ] }
        		\big|
        			\nu 
        				\{ 
        					y
        				:
        					y \prec x
        				\} 
        		-
        			\nu 
        				\{ 
        					y
        				:
        					\phi_k( y ) < \phi_k( x )
        				\} 
        		\big|
        			\le
            				\Delta_k
        				+
        					\sup_{ J \in \mc A_1 }
        						|
        							\Leb( J )
        						-
        							\nu( J )
        						|
        				.
    	 $
	\end{enumerate}
Consequently, the following convergence holds on $ \mc E $:
	\begin{enumerate}[ label = (\roman*), resume ]
		\item
		\label{claim convergence of empirical measures over intervals in the new order}
		$
			\displaystyle
    		\sup_{ x \in [ 0, 1 ] }
    		 	\left|
    				\phi( x )
    			- 
    				{P_n}
						\{
							y
						:
							y \prec x
						\}
    			\right|
    				\longrightarrow
    					0
		$
		as $ n \to \infty $.
    \end{enumerate}

\end{coro}

\begin{proof}
    Let $ k \ge 1 $ and $ \nu $ be a probability measure on $ [ 0 , 1 ] $.
    Using Proposition \ref{prop characterization of discrete order}, we have that
    \begin{multline*}
    	\left|
    		\nu 
    			\{ 
    				y
    			:
    				y \prec x
    			\} 
    	-
    		\nu 
    			\{ 
    				y
    			:
    				\phi_k( y ) < \phi_k( x )
    			\} 
    	\right|
    	 =
    			\nu 
    				\{ 
    					y
    				:
    					y \prec x, 
    					\,
    					i_{ x, y } > k
    				\} 
    		\\
    		\le
    			\nu 
    				\{ 
    					y
    				:
    					i_{ x, y } > k
    				\} 
    		=
    			\nu 
    				\{ 
    					y
    				:
    					\phi_k( y ) = \phi_k( x )
    				\} 
    	 =
        			\nu( 
        				\phi_k^{ -1 }( \phi_k( x ) )
    					).
    \end{multline*}

    \noindent
  From Proposition \ref{prop characterization of discrete order}\ref{claim intervals in step function}, 
    $
        \phi_k^{ -1 }( \phi_k( x ) )
        ,
    $
    is an interval between two consecutive values in the set $\{U_{-1},U_0,\cdots,U_k\}$.
    Therefore
    \begin{align*}
    \Leb(        			\phi_k^{ -1 }( \phi_k( x ) )
    					) \le \Delta_k \text{ and } 
					|\nu( 
        				\phi_k^{ -1 }( \phi_k( x ) )
    					)
        		-
        			\Leb(
        				\phi_k^{ -1 }( \phi_k( x ) )
    					)|\le \sup_{ J \in \mc A_1 }
        						|
        							\Leb( J )
        						-
        							\nu( J )
        						|,
    \end{align*}
establishing \ref{claim discretization estimate}.
    Taking $ \nu = \Leb $ and $ \nu = {P_n} $ in \ref{claim discretization estimate} gives us that, for $k \ge 1$,
		\begin{equation}
			\label{bound for phi_k approximation}
    		\displaystyle
			\sup_{ x \in [ 0, 1 ] }
        		\big|
        			\Leb 
        				\{ 
        					y
        				:
        					y \prec x
        				\} 
        		-
        			\Leb 
        				\{ 
        					y
        				:
        					\phi_k( y ) < \phi_k( x )
        				\} 
        		\big|
        			\le
            				\Delta_k
    	 \end{equation}

    \noindent
    and, for $n,k \ge 1$,
		$$
    		\displaystyle
			\sup_{ x \in [ 0, 1 ] }
        		\big|
        			{P_n}
        				\{ 
        					y
        				:
        					y \prec x
        				\} 
        		-
        			{P_n}
        				\{ 
        					y
        				:
        					\phi_k( y ) < \phi_k( x )
        				\} 
        		\big|
        			\le
            				\Delta_k
        				+
        					\sup_{ J \in \mc A_1 }
        						|
        							\Leb( J )
        						-
        							{P_n}( J )
        						|
				.
    	 $$
    
    \noindent
    Note that $ \mc A_1 $ can be replaced by $ \mc A_k $ above since $ \mc A_1 \subset \mc A_k $ for all $ k \ge 1 $.
    Observe also that the set
	$
		\{ 
			y
		:
			\phi_k( y ) < \phi_k( x )
		\} 
	$
	lies in $ \mc A_k $ for every $ x \in [ 0, 1 ] $ (see Proposition \ref{prop characterization of discrete order}\ref{claim intervals in step function}), from which we obtain the following inequality (for $n,k \ge 1$):
		$$
    		\displaystyle
			\sup_{ x \in [ 0, 1 ] }
        		\big|
        			\Leb 
        				\{ 
        					y
        				:
        					\phi_k( y ) < \phi_k( x )
        				\} 
        		-
        			{P_n} 
        				\{ 
        					y
        				:
        					\phi_k( y ) < \phi_k( x )
        				\} 
        		\big|
        			\le
        					\sup_{ J \in \mc A_k }
        						|
        							\Leb( J )
        						-
        							{P_n}( J )
        						|
					.
    	$$
	
		\noindent
		Collecting these bounds, we obtain the estimates (again for  $n,k \ge 1$)
		$$
			\displaystyle
    		\sup_{ x \in [ 0, 1 ] }
    		 	\big|
    				\Leb \{
							y
						:
							y \prec x
						\}
    			- 
    				{P_n}
						\{
							y
						:
							y \prec x
						\}
    			\big|
        			\le
            				2 \, \Delta_k
        				+
        					2 \sup_{ J \in \mc A_k }
        						|
        							\Leb( J )
        						-
        							{P_n}( J )
        						|
					.
		$$
		
	\noindent
	Since $\Leb \{ y :y \prec x \} =\phi(x)$,
	applying Proposition~\ref{prop consequences Glivenko-Cantelli}
	(items~\ref{claim convergence of maximal gap} 
	and \ref{claim convergence of empirical measures over intervals}) 
	concludes the proof.
\end{proof}

Finally, we establish the main properties of the objects $ \prec $, $ \phi $, and $ \murec_p $.

\begin{prop}
	\label{properties of phi}
	On $ \mc E $, the following statements hold:
	\begin{enumerate}[ label = (\roman*) ]
		\item
		\label{claim total order}
		the relation $ \prec $ is a total order,
				
		\item
		\label{claim uniform convergence of phi_k}
		$ \phi_k \to \phi $ uniformly on $ [ 0, 1 ] $ as $ k \to \infty $,
		
		\item
		\label{claim continuity of phi}
		$ \phi $ is continuous on $ [ 0 , 1 ] \setminus \{ U_j \}_{ j \ge 1 } $, 

		\item
		\label{claim phi is measure preserving}
		$ \phi $ preserves the measure $ \Leb $,
		and thus $\mu_\phi$ is a permuton.
	\end{enumerate}
\end{prop}

\begin{proof}
	Take $ x \neq y $ in $ [ 0, 1 ] $ and suppose that $ \mc E $ occurs.
	From Proposition \ref{prop consequences Glivenko-Cantelli}, 
	we know that the maximal gap $\Delta_k$ of the sequence $ ( U_j )_{ j \le k } $
	tends to $0$, so the infinite sequence $ ( U_j )_{ j \ge 1 }$
	 must intersect the nonempty interval $ \mc I_{ x, y } $.
    It follows that $ x $ and $ y $ are comparable in the order $ \prec $. This establishes \ref{claim total order}.
   
The uniform convergence in \ref{claim uniform convergence of phi_k} follows immediately from Proposition~\ref{prop characterization of discrete order}\ref{claim alternative definition of phi_k},
Eq.~\eqref{bound for phi_k approximation} and Proposition~\ref{prop consequences Glivenko-Cantelli}\ref{claim convergence of maximal gap}.
    The continuity in \ref{claim continuity of phi} then follows from 
    the uniform convergence in  \ref{claim uniform convergence of phi_k}
    and from Proposition \ref{prop characterization of discrete order}\ref{claim intervals in step function}, which describes the continuity of each $ \phi_k $.

 It remains to prove \ref{claim phi is measure preserving}, i.e.~that $\phi$
 preserves the Lebesgue measure.
   Let $ b \in ( 0, 1 )$. From Proposition \ref{prop characterization of discrete order}\ref{claim alternative definition of phi_k} and \ref{claim intervals in step function}, we see that
     each $ \phi_k $ attains only finitely many values
     and that the gap between consecutive values is bounded by $\Delta_k$.
       In particular, defining $y_k$ such that $\phi_k(y_k)$ is the smallest element in the range
    of $\phi_k$ above $b$ (which exists for large $k$), we have $\lim_{k\to +\infty} \phi_k(y_k) =b$.
   Also, using Proposition \ref{prop characterization of discrete order}\ref{claim alternative definition of phi_k}, we have
       \begin{align*}
    	\phi_k( y_k )
			= 
    			\Leb\big( 
    					\big\{ 
    						z
    					: 
    						\phi_k( z ) < \phi_k( y_k )
    					\big\} 
    				\big)
			 = 
    			\Leb\big( 
    					\big\{ 
    						z
    					: 
    						\phi_k( z ) \le b
    					\big\} 
    				\big).
    \end{align*}

	\noindent
	Recalling that $ \phi_k \nearrow \phi $, we have that
    \begin{multline*}
    	( \phi_\# \Leb )(  [ 0, b ] )
			 = 
    			\Leb\big( 
    					\big\{ 
    						z 
    					: 
    						\phi( z ) \le b
    					\big\} 
    				\big)
			 = 
    			\lim_{ k \to \infty } 
				\Leb\big( 
						\big\{ 
    						z
    					: 
    						\phi_k( z ) \le b 
    					\big\} 
    				\big)\\
			 = 
    			\lim_{ k \to \infty } 
					\phi_k( y_k )
			 = 
    				b
			.
    \end{multline*}
        
    \noindent
    Since this holds for every $ b \in ( 0, 1 ) $, the measures $ \phi_\# \Leb $ and $ \Leb $ are the same, and \ref{claim phi is measure preserving} holds.
\end{proof}

\subsection{The permutations}
\label{ssec:construction_permutations}
We continue our construction by realizing the recursive separable permutations on our probability space.
As mentioned earlier, the strategy will be to identify the partial orders from Section \ref{ssec:strategy} as permutations.
Informally, we do this by viewing the intervals in a partial order as forming a \enquote{shape} that defines a permutation diagram. 
For example, the 2nd partial order illustrated in Figure~\ref{fig:construction_order} has the shape of the permutation $ 132 $ since  
	its leftmost interval $ [ 0, U_1 ) $ is placed at the lowest level, 
	its middle interval $ [ U_1, U_2 ) $ is placed at the highest level, 
	and
	its rightmost interval $ [ U_2, 1 ] $ is placed at an intermediate level.
Note that this shape can also be deduced from simply considering the left endpoints $ 0, U_1 $, and $ U_2 $, which satisfy $ 0 < U_1 < U_2 $ and $ 0 \prec U_2 \prec U_1 $.

Throughout this section, we will assume that the $(U_j)_{j \ge 1}$ are distinct and different from 0 and 1. 
This ensures that the random set
$
	E_n=
    	\{ 		
			U_j 
			,\,  0 \le j \le n - 1 
    	\}
$ has size $n$.
Equipping each $E_n$ with the natural order $<$ and the random order $\prec$, 
the above considerations prompt us to define
\begin{equation}
\label{eq def lambda n}
	\incPermutations_n
     	=
    		\Perm
    			\Big( 
    				E_n, <,\prec 
    			\Big)
    	,
    		\qquad
    		n \ge 1
    	.
\end{equation}
An illustration is provided in Figure~\ref{fig:lambda}, which should be compared with Figures~\ref{fig:construction_order} and \ref{fig:graph_phik}.

In Proposition~\ref{prop these are recursive separable permutations}, we confirm that this permutation sequence has the same distribution as $(\sigma^{(n),p})_{n \ge 1}$.
The proof uses some standard exchangeability properties of i.i.d.~random variables, implying for example that the relative order of the $(U_i)_{i \le n}$ is independent from the set of their values $\{U_i, i \le n\}$.
As a first step, we establish a recursion for $ \incPermutations_n $, which will require some notation.

%
%
%
%
\begin{figure}
  \begin{center}
    \includegraphics[width=.3\textwidth]{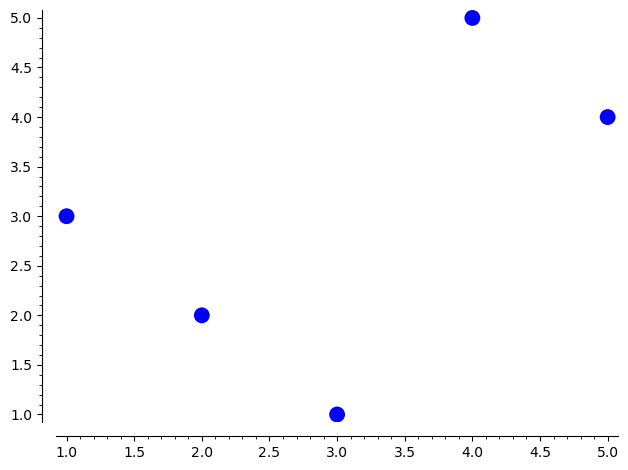} \quad \includegraphics[width=.3\textwidth]{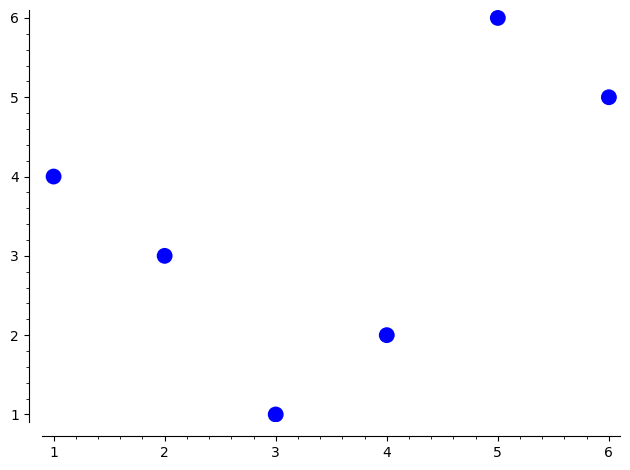} \quad \includegraphics[width=.3\textwidth]{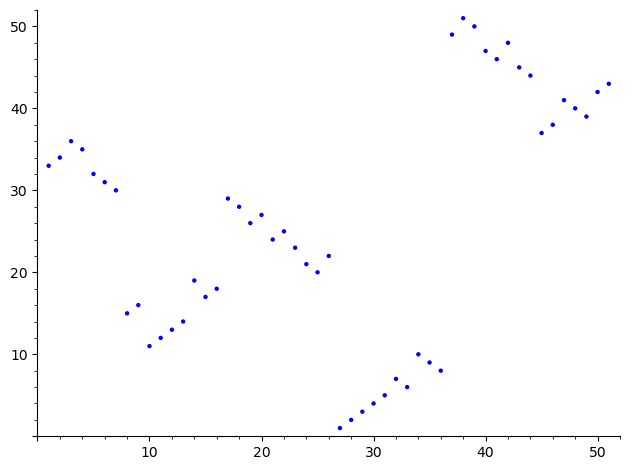}
  \end{center}
  \caption{The permutations $\lambda_n$ for $n \in \{5,6,50\}$, associated with our sampled sequences $(U_j)_{j \le 5}$
  and $(S_j)_{j \le 6}$, whose first values are given in Eq.~\eqref{Eq:Values_US}.}
  \label{fig:lambda}
\end{figure}
\medskip

Given a permutation $ \tau $ of $ n $, a sign $ s \in \{ \oplus, \ominus \} $, and an integer $ k \in \{ 1, \ldots, n \} $, let $ \tau_{s,k} $ be the permutation obtained by writing $ \tau $ in one-line notation, increasing all values bigger than $j= \tau(k) $, and replacing $ j $ by $ j \, j\!\! +\!\! 1 $ if $ s = \oplus $ or by $ j\!\! +\!\! 1 \, j $ if $ s = \ominus $.
For a permutation $\sigma$ of $n$ and an integer $k$,
we denote by $ \sigma^{\downarrow k}$ the pattern induced by $\sigma$
on the set of positions $\{1,\dots,n\} \setminus \{k\}$ (in other words,
we erase $\sigma_k$ in the one-line notation and decrease by 1 all values bigger than $\sigma_k$
to get a permutation).

With this notation in hand,
it can be verified that $  \tau_{s,k} $ is the unique permutation $ \sigma $ of $n+1$ satisfying the following properties:
	\begin{enumerate}[ label = (\roman*) ]
		\item
		$ \sigma^{\downarrow k}$ 
		and  $\sigma^{\downarrow (k+1)}$ are both equal to $ \tau$; 
		
		\item
		$ \sigma( k + 1 ) > \sigma(k)$ if $s=\oplus$ (resp.~$ \sigma( k + 1 ) < \sigma(k)$ if $s=\ominus$).
	\end{enumerate}

\begin{prop}
	\label{prop permutation recursion}
Suppose that the elements $(U_j)_{j=1}^n$ are distinct.
Let $ R_n $ denote the rank of $ U_n $ in 
	$ 
		\{
			U_j
		\}_{ j = 1 }^n
	$
	(i.e.~the unique integer satisfying
	$
		U_{ ( R_n, \, n ) }
			=
				U_{ n }
			).
	$
	The following recursion holds:
	$$
		\incPermutations_{ n + 1 }
			=
				( \incPermutations_n)_{S_n,R_n}
			,
				\qquad
				n \ge 1
			.
	$$
\end{prop}
Before proving the result, let us consider an example using the sampled values $(U_j)_{j \le 5}$ and $(S_j)_{j \le 5}$ in \cref{Eq:Values_US}.
Taking $ n = 5 $, the relevant quantities are $S_5=\oplus$, $R_5=3$, and the permutations $\lambda_5 $ and $\lambda_6 $ depicted in Figure~\ref{fig:lambda}.
It can be seen from the diagrams that the identity $\lambda_6=(\lambda_5)_{\oplus,3}$ does hold: $\lambda_6$ is obtained from $\lambda_5$ by an increasing inflation of its third point from the left.
\begin{proof}
	We will verify that $ \incPermutations_{ n + 1 } $ satisfies the three conditions that characterize
	$
    ( \incPermutations_n)_{S_n,R_n}
	$.
	Writing $k=R_n$, we note that $U_n$ corresponds to the $ (k+1 )$-st point
	from the left in $\lambda_{n+1}$ 
	(note that the index $j$ starts from $0$ in eq.~\eqref{eq def lambda n}).
	Therefore we have 
	$$(\lambda_{n+1})^{\downarrow k+\! 1}=\Perm
    			\big( 
    				\{ 
							U_j, 		
					  \, 0 \le j \le n - 1 		
    				\},<,\prec
    			\big)=\lambda_n.$$
	
	Let us consider the pattern $(\lambda_{n+1})^{\downarrow k}$.
		The $k$-th point from the left in $\lambda_{n+1}$ corresponds to $U_{(k-1,n)}$,
		 so that 
	\[(\lambda_{n+1})^{\downarrow k}= \Perm
    			\big( 
    				\{ 	
							U_j, \, 0\le j  \le n 
    				\} \setminus \{U_{(k-1,n)}\}, <, \prec 
    			\big).\]
			To show that this is also $\lambda_n$, we need to show that
$U_n$ and $U_{(k-1,n)}$ compare in the same way with other $U_j$'s ($j \le n$),
both for the natural order $<$ and for the random order $\prec$.

The case of the natural order $<$ is trivial by definition of the ordered statistics, since $U_n=U_{(k,n)}$.
Consider the random order $\prec$.
Since $U_n=U_{(k,n)}$, the element $U_n$ is the $U$ with smallest index
 in the interval $(U_{(k-1,n)},U_n]$, i.e.~$i_{U_{(k-1,n)},U_n}=n$.
 On the other hand, for $j <n$ and $U_j \neq U_{(k-1,n)}$, the interval $\mc I_{U_{(k-1,n)},U_j}$
 contains either $U_j$ or $U_{(k-1,n)}$, so that $i_{U_{(k-1,n)},U_j}<n$.
Applying Lemma~\ref{lem extrema and the random order},
we then find that
	$
		i_{ U_{(k-1,n)}, U_j }
			=
				i_{ U_j, U_n }
			<
				\infty
			,
	$
and
$U_{(k-1,n)}$ and $U_n$ compare in the same way with $U_j$ in the order $\prec$,
which is what we needed to show.
Thus $\lambda_{n+1}$ verifies the first condition in the characterization
of $
    ( \incPermutations_n)_{S_n,R_n}
	$.

It remains to show that $ \lambda_{n+1}( k) < \lambda_{n+1}(k+1)$ if $S_n=\oplus$,
 (resp.~$ \lambda_{n+1}( k ) > \lambda_{n+1}(k+1)$ if $S_n=\ominus$).
 We consider the case $S_n=\oplus$.
 Since the $k$-th and $ ( k+1 )$-st points from the left in $ \lambda_{n+1}$
 correspond to $U_{(k-1,n)}$ and $U_n$ respectively,
 it suffices to show that $U_{(k-1,n)} \prec U_n$.
 This  follows from the fact that $i_{U_{(k-1,n)},U_n}=n$ (see above) and 
 the assumption $S_n=\oplus$
 (recall that $U_{(k-1,n)} < U_n$ for the natural order).
  
 We have proved that $\lambda_{n+1}$ verifies the second condition in the characterization
of $
    ( \incPermutations_n)_{S_n,R_n}
	$, implying $\lambda_{n+1}= ( \incPermutations_n)_{S_n,R_n}$, as desired.
\end{proof}

\begin{prop}
	\label{prop these are recursive separable permutations}
	The permutations $ ( \incPermutations_n )_{ n \ge 1 } $ are recursive separable permutations.
	In other words, 
	$$
		(\incPermutations_n)_{n \ge 1}
			\stackrel{d}=
				(\sigma^{(n),p})_{n \ge 1}.
	$$
\end{prop}
\begin{proof}
	We start with the following observation.
	\begin{itemize}
	\item To compare $U_j$ and $U_k$ in the order $\prec$, we need to look
	at the sign $S_{i_{U_j,U_k}}$. If $j,k <n$ and $U_j \ne U_k$,
	then the interval $\mc I_{U_j,U_k}$ contains either $U_j$ or $U_k$,
	and thus $i_{U_j,U_k}<n$. Therefore, the restriction of $\prec$ to $\{U_0,\dots,U_{n-1}\}$
	only depends on the restriction of the natural order $<$ to  $\{U_0,\dots,U_{n-1}\}$
	and on the signs $S_1, \ldots, S_{ n - 1 } $.
	Consequently, the tuple of permutations $(\incPermutations_1, \dots, \incPermutations_n) $
	only depends on this data as well.
	\item On the other hand, $ R_n $ describes how $ U_n $ fits into the
	gaps of the set $\{U_0,\dots,U_{n-1}\}$, and is independent
	of its order structure.
	\end{itemize}	
	Therefore the random variables $(\incPermutations_1, \dots, \incPermutations_n) $, 
	$ S_n $, and $ R_n $ are independent.
	Combining this with Proposition \ref{prop permutation recursion}, 
	this implies that $(\incPermutations_n)_{n \ge 1}$ is a Markov process
	and we can compute its transition kernel.
	If $ \pi $ and $\tau$ are permutations of $ n $ and $n+1$ respectively,
	one has
	\begin{align*}
		\prob( 
				\incPermutations_{ n + 1 } = \tau | \incPermutations_n= \pi
			)
				& = 
                						\sum_{ j = 1 }^n
						\sum_{ s \in \{ \oplus, \ominus \} }
						\indicator( 
                				\pi_{s,j} = \tau
                			)
                                    \prob( 
									R_n = j,
								S_n = s 
	                			)
				\\
			&= 
                								\sum_{ j = 1 }^n
						\left(\frac{p}{n}
						\,
						\indicator( 
                				\pi_{\oplus,j} = \tau
                			)
					+
						\frac{1-p}{n}
						\,
						\indicator( 
                				 \pi_{\ominus,j} = \tau
                			) \right)
				\\
				& = 
						\prob( 
                				\sigma^{ ( n+1), \, p  } = \tau
							|
                				\sigma^{ ( n), \, p  } = \pi
                			)
	\end{align*}
Since $(\incPermutations_n)_{n \ge 1}$ and $(\sigma^{ ( n), \, p  })_{n \ge 1}$
are both Markov chains with the same transition kernel and the same initial distribution
($\incPermutations_1$ and $\sigma^{ ( 1), \, p  }$ are both equal to the 
unique permutation of $1$ a.s.), they have the same distribution.
\end{proof}

The following result is the last one in this section.
Here, we show that our permutation model satisfies another type of recursion,
referred to as {\em consistency}.
It is not needed for the convergence argument, 
but it is useful in the next section for studying 
the expected pattern densities of $\murec_p$.

\begin{prop}
	\label{prop consistency of permutations}
	The sequence $(\incPermutations_n)_{n \ge 1} $ is a consistent family
	of random permutations,
	in the sense of \cite[Definition 2.8]{bassino2020universal}:
	namely, for $ n \ge 1 $, the permutation obtained by removing a uniformly random point from $ \lambda_{ n + 1 } $ is distributed as $\lambda_n$.
\end{prop}

\begin{proof}
	To avoid repetition, let us take the following convention throughout the proof:
    each new random variable introduced in this proof should be assumed to be independent of all previously defined random variables.

	The $ n = 1 $ case of the proposition is trivial.
	Proceeding by induction, we fix $ n > 1 $ and assume that the result holds for $ n - 1 $.
	Let $ j $ be a uniformly random integer in $ \{ 1, \ldots, n + 1 \} $
	and
	$ A $ be the event
	$
			\{
				R_n \le j \le R_n + 1
			\}
		.
	$
	As we will show later, the following statements hold:
	\begin{enumerate}[ label = (\roman*) ]
		\item \label{item:lambda_ind_A}
		$ \lambda_n $ is independent of $ A $,
		\item \label{item:on_A}
		$ 
			\lambda_{ n + 1 }^{ \downarrow  j }
				=
					\lambda_n
		$
	 	on $ A $,
		and
		\item \label{item:on_Ac}
		conditionally given $ A^c $, 
		$ 
			\lambda_{ n + 1 }^{ \downarrow  j }
		$
		is distributed as
		$ 
			\lambda_n.
		$
	\end{enumerate}

	\noindent
	The result follows from these statements. Indeed,
	if $ \tau $ is a permutation of $ n $, then
	\[
			\prob(
				\lambda_{ n + 1 }^{ \downarrow  j }
					=
						\tau
			)
				 = 
						\prob(
            				\lambda_{ n + 1 }^{ \downarrow  j }
            					=
            						\tau
						\vert
							A^c						
						)
						\prob( A^c )
					+
						\prob(
            				\lambda_{ n + 1 }^{ \downarrow  j }
            					=
            						\tau
						\vert
							A						
						)
						\prob( A ).
						\]
		Provided that \cref{item:on_Ac} holds, we have
		$\prob(
            				\lambda_{ n + 1 }^{ \downarrow  j }
            					=
            						\tau
						\vert
							A^c						
						)=
						\prob(
            				\lambda_n
            					=
            						\tau				
						)
				$, while \cref{item:lambda_ind_A,item:on_A} would imply
				$\prob(
            				\lambda_{ n + 1 }^{ \downarrow  j }
            					=
            						\tau
						\vert
							A						
						)=\prob(
            				\lambda_ n             					=
            						\tau
						\vert
							A						
						)=\prob(
            				\lambda_ n             					=
            						\tau)$.
	Therefore, given \cref{item:lambda_ind_A,item:on_A,item:on_Ac},
	one has 	
				\[
			\prob(
				\lambda_{ n + 1 }^{ \downarrow  j }	=\tau)
				 = 
						\prob(
            				\lambda_n
            					=
            						\tau				
						)
						\prob( A^c )
					+
						\prob(
            				\lambda_n
            					=
            						\tau					
						)
						\prob( A )
				 = 
						\prob(
            				\lambda_n
            					=
            						\tau				
						)
				.
 \]

	It remains, then, to verify the above statements.
	As explained in the proof of Proposition \ref{prop these are recursive separable permutations}, $\lambda_n$ is independent from $R_n$. Moreover, by construction,
	$j$ is independent from $(\lambda_n,R_n)$. Thus $\lambda_n$ is independent
	from $(j,R_n)$ and hence from the event $A$, proving \cref{item:lambda_ind_A}.

    Moreover it follows from Proposition~\ref{prop permutation recursion}
    that, a.s., $\lambda_{ n + 1 }^{ \downarrow  R_n }=\lambda_{ n + 1 }^{ \downarrow  R_n\! +\!1}=\lambda_n$, implying  \cref{item:on_A}.
	For the third statement, we introduce the random variables
	\begin{align*}
		X
			& = 
				j - \indicator( j > R_n )  
			,
			\\
		Y
			& = 
				R_n - \indicator( R_n > j )  
			.
	\end{align*}
	Using Proposition~\ref{prop permutation recursion} and an obvious commutation relation between
    the operator $\downarrow j$ and $(\cdot)_{s,k}$, we have
	$$
		\lambda_{ n + 1 }^{ \downarrow  j }
        = \big( (\lambda_n)_{S_n,R_n} \big)^{ \downarrow  j } =
            ( \lambda_n^{ \downarrow X })_{S_n,Y},
			\quad
				\text{on } A^c
			.
	$$

	\noindent
	Notice also that
	$ S_n $, $ \lambda_n $, $ X $, and $ Y $ are mutually independent given $ A^c $, 
	$ S_n $ and $ \lambda_n $ are independent of $ A^c $,
    $ X $ is uniformly distributed in $\{1,\ldots, n\} $ conditionally given $ A^c $, 
	and 
    $ Y $ is uniformly distributed in $\{1,\ldots, n - 1\} $ conditionally given $ A^c $.
	Therefore, the conditional distribution of $ ( S_n, \lambda_n, X, Y ) $ given $ A^c $ 
	is exactly the distribution of
	$ ( S_n, \lambda_n, k, \ell ) $, 
	where
    $ k $ is a uniformly random integer in $ \{1,\dots, n \} $
	and
    $ \ell $ is a uniformly random integer in $\{1,\dots, n - 1 \} $.
	Recalling from Proposition \ref{prop these are recursive separable permutations} that
	$ \incPermutations_{ n - 1 } $, $ S_{ n - 1 } $, and $ R_{ n - 1 } $ are mutually independent and making use of the induction hypothesis, we find that conditionally given $ A^c $, 
    \[
		\lambda_{ n + 1 }^{ \downarrow  j }
		= (\lambda_n^{ \downarrow X })_{S_n,Y}
			 \stackrel{d} =
            ( \lambda_n^{ \downarrow k } )_{S_n,\ell}
			 \stackrel{d} =
             ( \lambda_{ n - 1 } )_{S_{n-1},\ell}
			\stackrel{d} =
				\lambda_n
			.
            \]
	\noindent
	This establishes the third fact and concludes the proof.
\end{proof}

\subsection{The convergence argument}
\begin{proof}[Proof of Theorem \ref{thm:convergence}]
        We recall that the event \[\mc E
		=
			\Bigg
			\{
				( U_j )_{ j \ge 1 }
				\text{ are distinct},
				\,
			\sup_{ x \in [ 0, 1 ] }
			 	\bigg|
					x 
				- 
					\frac{1}{n} 
					\sum_{ j = 1 }^n 
						\indicator( U_j \le x )
				\bigg|
						\xrightarrow[ n \to \infty ]{}
							0
			\Bigg
			\}\]
	holds almost surely.
        Also, since the permutations $ ( \incPermutations_n )_{ n \ge 1 } $ are recursive separable permutations (see Proposition~\ref{prop these are recursive separable permutations}), 
        it suffices to show the permuton convergence 
        $$
        	\mu_{ \incPermutations_n }
        		\xrightarrow[ n \to \infty ]{}
        			\murec_p
        			\quad
        			\text{on } \mc E
        		.
        $$
        
        \noindent
        Making use of (\ref{remark describing permutations as pushforwards}), the identity
        $
        	\murec_p
        		=
        			( \, \cdot \,, \phi( \, \cdot \, ) )_{ \# }
        			\Leb
        		,
        $
		and Proposition \ref{prop characterization pushforward convergence}, we can reformulate this convergence as the convergence of functions
        $$
        	f_{ \incPermutations_n }
        		\longrightarrow
        			\phi
        				\quad
            			\text{in } L^1[ 0, 1 ]
                			\quad
                			\text{on } \mc E
        		.
        $$
        
        \noindent
        We will establish this convergence by showing that 
        $
        	f_{ \incPermutations_n }
        		\to
        			\phi
        $
        pointwise almost everywhere in $ [ 0 , 1 ] $ whenever $ \mc E $ occurs.

	To this end, fix
	$ 
       	x 
       		\in 
       			( 0, 1 ]
       			\setminus 
       			\{
       				U_j
       			\}_{ j \ge 1 }
	$
	and suppose that $ \mc E $ occurs.
	Since $ x > 0 $, the quantity $ \ceiling{ n x } $ lies in $\{ 1, \ldots, n \} $ and we can write
	\begin{multline}
	\label{eq:diff_flambda_phi}
		\left|
        	f_{ \incPermutations_n }( x )
		-
			\phi( x )
		\right|
			 \le 
            		\left|
                    	f_{ \incPermutations_n }( x )
            		-
        				P_n
            				(\left\{ 
            					y
            				:
            					y \prec U_{ ( \ceiling{ n x } - 1, n - 1 ) }
            				\right\}) 
            		\right|
			\\
			 				+
            		\left|
        				P_n
            				(\left\{ 
            					y
            				:
            					y \prec U_{ ( \ceiling{ n x } - 1, n - 1 ) }
            				\right\}) 
            		- 
            			\phi( U_{ ( \ceiling{ n x } - 1 , n - 1 ) } )
            		\right|
			 				+
            		\left|
            			\phi( U_{ ( \ceiling{ n x } - 1 , n - 1 ) } )
            		-
            			\phi( x )
            		\right|
			.
	\end{multline}

	\noindent
	Corollary \ref{cor discretization estimate}\ref{claim convergence of empirical measures over intervals in the new order} tells us that the second term above will converge to zero as $ n \to \infty $.
	The third term also converges to zero: 
	indeed, Proposition \ref{prop consequences Glivenko-Cantelli} implies that
	\begin{align*}
		\left|
			x
		-
			U_{ ( \ceiling{ n x } - 1, n - 1 ) }
		\right|
			& \le
            		\left|
            			x
            		-
            			U_{ ( \ceiling{ n x }, n - 1 ) }
            		\right|
				+
            		\left|
            			U_{ ( \ceiling{ n x }, n - 1 ) }
            		-
            			U_{ ( \ceiling{ n x } - 1 , n - 1 ) }
            		\right|
			\\
			& \le
            		\left|
            			x
            		-
            			U_{ ( \ceiling{ n x }, n - 1 ) }
            		\right|
				+
            		\Delta_{ n - 1 }
			\\
			& \quad \xrightarrow[ n \to \infty ]{}
				0
			,
	\end{align*}
	\noindent
	and $ \phi $ is continuous at $ x $ (see Proposition \ref{properties of phi}\ref{claim continuity of phi}). 
	It only remains, then, to show that the first term in the upper bound \eqref{eq:diff_flambda_phi} also converges to zero.
    This follows from the definition of $\incPermutations_n$ and Eq.~\eqref{eq formula sigma}, which allows us to write
    \begin{align*}
    	f_{ \incPermutations_n }( x )
    		 =
    			\frac 	{
    						\incPermutations_n( \ceiling{ n x } )
    					}{
    						n
    					}
			 &=
        			\frac{ 1 }{ n }
    			+
    				\frac{ 1 }{ n }
        			\sum_{ j = 0 }^{ n - 1 }
        				\indicator
        					\left( 
                            	U_j
                        			\prec 
                        				U_{ ( \ceiling{ n x } - 1, \, n - 1 ) } 
        					\right)\\
    		 &=
        			\frac{ 1 }{ n }
				+
    				{P_n}
        				(\left\{ 
        					y
        				:
        					y \prec U_{ ( \ceiling{ n x } - 1, n - 1 ) }
        				\right\}) 
			.
\end{align*}
Hence, whenever $\mc E$ occurs, for any 
$x \in ( 0, 1 ] \setminus \{U_j\}_{ j \ge 1 }$, the quantity $f_{ \incPermutations_n }( x )$
tends to $\phi(x)$. This concludes the proof.
\end{proof}

\section{Properties of the recursive separable permuton}
\label{sec:properties}

\subsection{Self-similarity}
\label{sec:SelfSimilarity}
{Fix $p \in(0,1)$. We recall informally the definition of $\Phi_p$ from the introduction.
If $\bm \mu$ is a random permuton, then $\Phi_p(\bm \mu)$
is the random permuton obtained as follows.}
\begin{itemize}
\item {We let $U$ and $S$ be independent random variables,
$U$ being uniform in $[0,1]$ and $S$ being $\oplus$ with probability $p$ and $\ominus$ with probability $1-p$.}
\item {We construct $\Phi_p(\bm \mu)$ by juxtaposing two independent
copies $\bm \mu_0$ and $\bm \mu_1$ of $\bm \mu$, where $\bm \mu_0$ is scaled by a factor $U$,
while $\bm \mu_1$ is scaled by a factor $1-U$. If $S=\oplus$, we put $\bm \mu_0$ below and on the left of $\bm \mu_1$, while, if   $S=\ominus$,
we put $\bm \mu_0$ above and on the left of $\bm \mu_1$.}
\end{itemize}
{The goal of this section is to prove Proposition~\ref{prop:self_similarity},
which states that the law of $\murec_p$ is the unique probability distribution invariant by $\Phi_p$.}

{We first prove the uniqueness of such an invariant probability distribution.}
For this, we see, with a small abuse of notation, $\Phi_p$ as a map from
the set $\mathcal M_1(\mathcal P)$ of probability measures on the set $ \mc P $ of permutons to itself.
Both $\mathcal P$ and $\mathcal M_1(\mathcal P)$ 
are endowed with the first Wasserstein distance; to avoid confusion, we will use a boldface notation $\bdW$ for the distance on $\mathcal M_1(\mathcal P)$
and a standard $d_W$ for that on $\mathcal P$.

\begin{lemm}
The map $\Phi_p:\mathcal M_1(\mathcal P) \to \mathcal M_1(\mathcal P)$
is a contraction with Lipschitz constant at most $2/3$.
\label{lem:PhiP_contraction}
\end{lemm}
\begin{proof}
Let $\mathbb P$ and $\mathbb Q$ be two probability measures
on the set $\mathcal P$ of permutons and call $d:=\bdW(\mathbb P,\mathbb Q)$ their Wasserstein distance.
Fix $\eps>0$. By definition of Wasserstein distance,
one can find random permutons $\bm\mu$ and $\bm\nu$ on the same probability space
such that $\bm\mu$ has distribution $\mathbb P$, $\bm\nu$ has distribution $\mathbb Q$
and \[\mathbb E(d_W(\bm\mu, \bm\nu)) <d+\eps.\]
We consider two independent copies $(\bm\mu_0,\bm\nu_0)$ and $(\bm\mu_1,\bm\nu_1)$
of the pair $(\bm\mu,\bm\nu)$. We also consider a single pair $(U,S)$,
where $U$ is a uniform random variable in $[0,1]$ and $S$ a random sign with $\mathbb P(S=\oplus)=p$, such that $U$ and $S$ are independent from each other and from 
$(\bm\mu_0,\bm\nu_0)$ and $(\bm\mu_1,\bm\nu_1)$.
By definition of $\Phi_p$, the measure $\Phi_p(\mathbb P)$ (resp. $\Phi_p(\mathbb Q)$) 
is the distribution
of the random permuton $\bm\mu_0 \otimes_{(U,S)} \bm\mu_1$ 
(resp.~$\bm\nu_0 \otimes_{(U,S)} \bm\nu_1$). Therefore
\begin{equation}
\label{eq:bouding_distances_afterPhiP}
 \bdW(\Phi_p(\mathbb P),\Phi_p(\mathbb Q))
\le \mathbb E \big( d_W(\bm\mu_0 \otimes_{(U,S)} \bm\mu_1, 
\bm\nu_0 \otimes_{(U,S)} \bm\nu_1) \big).
\end{equation}
It is straightforward to see that, a.s.,
\begin{equation}
\label{eq:boud_dW_otimesUS}
d_W\big( \bm\mu_0 \otimes_{(U,S)} \bm\mu_1, 
\bm\nu_0 \otimes_{(U,S)} \bm\nu_1 \big) \le U^2 d_W\big( \bm\mu_0,\bm\nu_0\big)
+(1-U)^2 d_W\big( \bm\mu_1,\bm\nu_1\big),
\end{equation}
where the factor $U^2$ is explained by the rescaling of distances and 
of weights of $\bm\mu_0$, both by a factor $U$ in the construction of $\bm\mu_0 \otimes_{(U,S)} \bm\mu_1$ (and similarly the factor $(1-U)^2$ comes from the rescaling
of distances and weigths in $\bm\mu_1$ by $1-U$).
From Eqs.~\eqref{eq:bouding_distances_afterPhiP} and~\eqref{eq:boud_dW_otimesUS},
 using the independence of $U$ from $(\bm\mu_0,\bm\nu_0)$ and $(\bm\mu_1,\bm\nu_1)$
 and the equality $(\bm\mu_0,\bm\nu_0)\stackrel{d}=(\bm\mu_1,\bm\nu_1)$, we get
\[ \bdW(\Phi_p(\mathbb P),\Phi_p(\mathbb Q))
\le \mathbb E\big(U^2 +(1-U)^2\big) \, \mathbb E\big(d_W\big( \bm\mu_0,\bm\nu_0\big) \big) < \tfrac23 (d+\eps).\]
Since this holds for any $\eps>0$, we have
\[ \bdW(\Phi_p(\mathbb P),\Phi_p(\mathbb Q)) \le \tfrac23 \, d=  \tfrac23 \bdW(\mathbb P,\mathbb Q).\]
This proves the lemma.
\end{proof}

Lemma~\ref{lem:PhiP_contraction} implies in particular the existence and uniqueness
of a probability measure $\mathbb P_p$ on the set of permutons such that
$\mathbb P_p =\Phi_p(\mathbb P_p)$.
We still need to identify $\mathbb P_p$ with the distribution of the recursive separable
permuton $\murec_p$.
To do this, we first show that the recursive separable {\em permutations} exhibit self-similarity 
and then use Theorem~\ref{thm:convergence} 
to carry this property over to the limit.
It might be possible to establish the self-similarity of the permuton directly, 
but we think that the self-similarity of the permutations is of its own interest.

Below, we denote the distribution of a random variable $ X $ by $\Law( X ) $.
In addition, we use the notion of skew sums and direct sums of permutations introduced in Section~\ref{ssec:pattern_densities_intro}.
 \begin{prop}
 \label{prop:self_similarity_discrete}
  Let $(\sigma^{(n),p})_{n\ge 1}$, $(\tau^{(n),p})_{n\ge 1}$ and $(\rho^{(n),p})_{n\ge 1}$
   be independent copies of the permutation process defined in Section~\ref{ssec:model},
  and let $I$ be a uniform integer in $\{1,\dots,n-1\}$.
  Then, for any fixed $n \ge 1$, we have
  \[ \Law(\sigma^{(n),p}) = p\, \Law(\tau^{(I),p} \oplus \rho^{(n-I),p} )
  +(1-p) \Law(\tau^{(I),p} \ominus \rho^{(n-I),p} ).\]
 \end{prop}
 We note that this is an equality for fixed $n \ge 1$, and it does not
 extend to a recursive description of the law of the process $(\sigma^{(n),p})_{n\ge 1}$.
 \begin{proof}
 For $n=2$, we have $\sigma^{(2),p}=12$ with probability $p$
 and $\sigma^{(2),p}=21$ with probability $1-p$.
 We will show that conditionally on $\sigma^{(n),p}=12$,
 we have 
 \begin{equation}
 \label{eq:Law_SigmaN_Knowing_Sigma2}
 \Law(\sigma^{(n),p} | \sigma^{(2),p}=12) = 
 \Law(\tau^{(I),p} \oplus \rho^{(n-I),p} ).
 \end{equation}
 We recall that $(\sigma^{(n),p})_{n\ge 1}$ is defined recursively via inflation operations,
 where one point is replaced by two adjacent points
 (either in ascending or descending positions).
 In such operations, we will think at the two new points as the \enquote{children}
  of the point they replace.
  This defines an ascendant/descendant relation on the set of points of all $\sigma^{(n),p}$
  for $n \ge 1$.
  
  In particular, points in $\sigma^{(n),p}$ can be split in two parts,
 defined as the descendants of the points $1$ and $2$ in $\sigma^{(2),p}$.
 Let us call $\tau$, resp.~$\rho$, the pattern formed the descendants of $1$, resp.~$2$.
 Since points are always replaced by pairs of adjacent points,
  when $\sigma^{(2),p}=12$, we have $\sigma^{(n),p} = \tau \oplus \rho$.
Conditionally of its size, which we call $k$, the permutation $\tau$
has the same distribution as $\tau^{(k),p}$ since it is obtained from $1$ 
by successive random inflations. Similarly, $\rho$ has the same distribution as $\rho^{(n-k),p}$. Moreover, both are independent conditionally on $k$.

Therefore, the only remaining thing to be proven in order to establish~\eqref{eq:Law_SigmaN_Knowing_Sigma2}
is that $k$ is uniformly distributed in $\{1,\dots,n-1\}$.
Letting $I_j$ being the number of descendants of $1$ in $\sigma^{(j),p}$,
we see that $(I_j)_{j \ge 1}$ has the Markov property and 
 \[
 \begin{cases}
\mbb P(I_j=k | I_{j-1}=k-1) = \frac{k-1}{j-1} \\
\mbb P(I_j=k | I_{j-1}=k) = \frac{j-1-k}{j-1}
\end{cases} 
 \]
Using this and the base case $I_2=1$ a.s., 
an immediate induction shows that, for any $j \ge 2$, the random variable
$I_j$ is uniformly distributed in $\{1,\dots,j-1\}$
(we remark that this is a basic model of P\'olya urn.)
This concludes the proof of Eq.~\eqref{eq:Law_SigmaN_Knowing_Sigma2}.

With similar arguments, one can prove
\begin{equation}
 \label{eq:Law_SigmaN_Knowing_Sigma2_Case2}
\Law(\sigma^{(n),p} | \sigma^{(2),p}=21) \stackrel{d}= 
 \Law(\tau^{(I),p} \ominus \rho^{(n-I),p} ).
 \end{equation}
 Since $\mbb P(\sigma^{(n),p}=12)=p=1-\mbb P(\sigma^{(n),p}=21)$,
 the proposition follows from Eqs.~\eqref{eq:Law_SigmaN_Knowing_Sigma2}
 and~\eqref{eq:Law_SigmaN_Knowing_Sigma2_Case2}.
 \end{proof}

\begin{proof}[Proof of Proposition~\ref{prop:self_similarity}]
Taking the limit $n \to +\infty$ in Proposition~\ref{prop:self_similarity_discrete} gives
$\murec_p \stackrel{d}= \Phi_p(\murec_p)$
(recall from Theorem~\ref{thm:convergence} that $\sigma^{(n),p}$ converge a.s., and hence
in distribution to $\murec_p$).
The uniqueness statement in Proposition~\ref{prop:self_similarity} follows 
from Lemma~\ref{lem:PhiP_contraction}.
\end{proof}

\subsection{Expected pattern densities}
\label{ssec:expected_pattern_densities}

\begin{proof}[Proof of Proposition~\ref{prop:expected_densities_murec}]
From Proposition~\ref{prop consistency of permutations}, we know that $(  \sigma^{(n),p})_{n \ge 1}$
is a consistent family of random permutations.
 By~\cite[Proposition 2.9]{bassino2020universal}, there exists a random permuton $\bm\mu$ such that  $\sigma^{(n),p}$ converge to $\bm\mu$ in distribution
 and, for each $n\ge 1$,
 the random permutation $\Sample(\bm\mu,n)$ has the same distribution as  $\sigma^{(n),p}$.
On the other hand, $\sigma^{(n),p}$ converges almost surely to $\murec_p$ (Theorem~\ref{thm:convergence}), so $\bm\mu$ and $\murec_p$ must be equal in distribution.
 Therefore,
 \[ \sigma^{(n),p} \stackrel{d}{=} \Sample(\bm\mu,n) \stackrel{d}{=} \Sample(\murec_p,n), 
 	\qquad n \ge 1.
 \]
 Together with \cite[Theorem 2.5]{bassino2020universal}, this implies that
 \[ \mathbb E\big[ \dens(\pi,\murec_p) \big] = \mathbb P\big[ \Sample(\murec_p,n)=\pi\big]= \mathbb P\big[ \sigma^{(n),p}=\pi \big],
 \qquad \pi \in S_n.
 \]
 This proves the first part of Proposition~\ref{prop:expected_densities_murec}.
 
 For the second part, we associate with a realization of the random sequence 
 $(  \sigma^{(k),p})_{k \le n}$ a (rooted binary increasing decorated) tree $T_n$ 
 such that $\Perm(T_n)=\sigma^{(n),p}$ (where we use the map $\Perm$
 from trees to permutations introduced in Section~\ref{ssec:pattern_densities_intro}).
 We proceed by induction on $n$.
 For $n=1$, the tree $T_1$ is reduced to a single leaf.
 For $n \ge 1$, assume that $T_n$ is constructed and 
 that in the sampling process for $\sigma^{(n+1),p}$ (Section~\ref{ssec:model}),
 we have chosen some integer $j$ and some sign $s$
 (where $s=\oplus$ means that $j$ is replaced by $j\, j\!+\!1$
 and $s=\ominus$ means that $j$ is replaced by $j\!+\!1 \, j$).
 Then $T_{n+1}$ is obtained by replacing the $j$-th leaf of $T_n$
 by an internal node with label $n$ and decoration $s$, with two children which are leaves.
 It can be verified, using the induction hypothesis $\Perm(T_n)=\sigma^{(n),p}$,
 that $\Perm(T_{n+1})=\sigma^{(n+1),p}$.
 Informally, $T_n$ encodes the history of the construction of $\sigma^{(n),p}$.
 An example is given on Figure~\ref{fig:history}.
 
 \begin{figure}
\[\includegraphics[width=\textwidth]{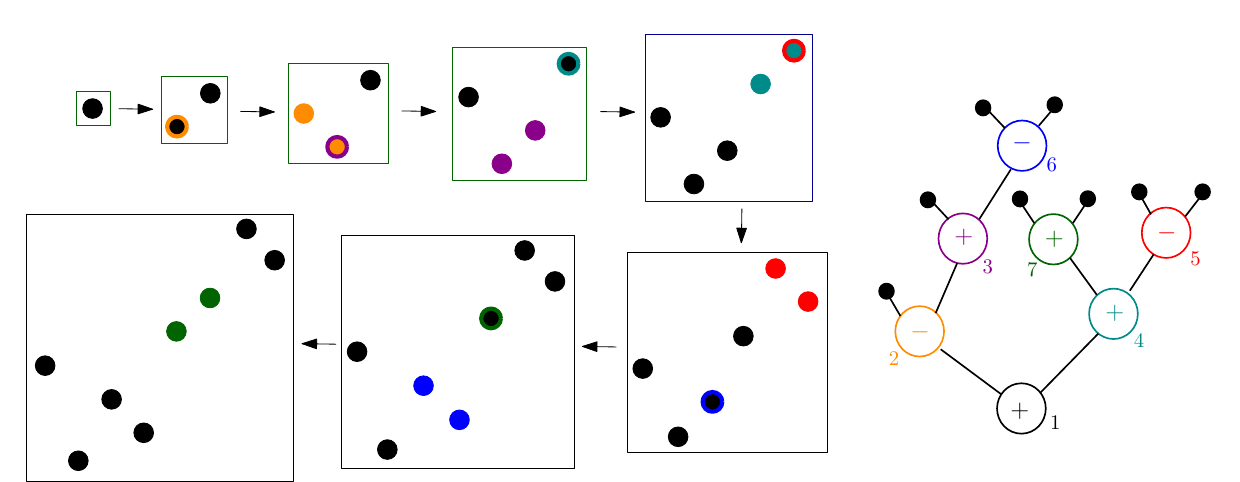}\]
\caption{A realization of the process $(\sigma^{(k),p})_{k \le 8}$ and the associated
tree $T_8$. Colors indicate points which have been selected for inflation at each step,
and the internal node created in the corresponding step in the construction of $T_n$.}
\label{fig:history}
\end{figure}
 
 The above construction yields a random tree $T_n$ living in the same probability space 
 as $(  \sigma^{(k),p})_{k \le n}$. For a fixed (rooted binary increasing decorated) tree $T$,
 the event $T_n=T$ amounts to making specific choices of integers $j$ and signs $s$ at each
 step of the construction of $(  \sigma^{(k),p})_{k \le n}$.
 Thus we have 
 \[ \mathbb P[T_n =T] = \frac{1}{(n-1)!} p^{\oplus(T)} (1-p)^{\ominus(T)},\]
 where $\oplus(T)$ and $\ominus(T)$ denote the number of $\oplus$ decorations
 and $\ominus$ decorations in $T$ respectively.
 Since $\Perm(T_n)=\sigma^{(n),p}$, this implies that for a given permutation $\pi$
 of $n$,
 \[ \mathbb P \big[ \sigma^{(n),p}=\pi \big] = \frac{1}{(n-1)!} \sum_{T: \Perm(T)=\pi} 
 p^{\oplus(T)} (1-p)^{\ominus(T)}.\]
 But an immediate inductive argument
 shows that whenever $ \Perm(T)=\si$, we have $\ominus(T)=\des(\sigma)$,
 and consequently $\oplus(T)=n-1-\des(\sigma)$.
 Therefore, all terms in the above sum are equal, and we obtain
 \[ \mathbb P \big[ \sigma^{(n),p}=\pi \big] = 
 \frac{N_{inc}(\pi)}{(n-1)!}\, (1-p)^{\des(\pi)}\, p^{n-1-\des(\pi)}.\qedhere\]
 \end{proof}

\subsection{Intensity measure}
\label{sec:Intensity}
Recall that the intensity measure of $\murec_p$, denoted $I \murec_p$,
is a probability measure on $[0,1]^2$ defined as follows:
for all Borel subsets $A$ of $[0,1]^2$, we have $I \murec_p(A)= \mathbb E[\murec_p(A)]$.
The goal of this section is to compute this intensity measure.
We start with a lemma.
\begin{lemm}
  \label{lem:Mickael}
 The intensity measure $I \murec_p$ of the recursive separable permuton
 is the distribution of $(U,\phi(U))$, where $U$ is a uniform random variable
 in $[0,1]$, $\phi$ is the random function of Section~\ref{sec:construction_permuton} ,
 and $U$ and $\phi$ are independent from each other.
\end{lemm}
\begin{proof}
{For any function $f$, we denote $f_2$ the map $x \mapsto(x,f(x))$.
Then, for any measurable set $A \subset [0,1]^2$,
we have
\[\mu_f(A) = \Leb\big(f_2^{-1}(A)\big) = \int_0^1 \One[(x,f(x)) \in A] dx = \mathbb E\big[ \One[(U,f(U)) \in A] \big],\]
where $U$ is a uniform random variable in $[0,1]$.
Recalling that $\murec_p=\mu_\phi$, where $\phi$ the random function of Section~\ref{sec:construction_permuton} ,
 we have
\[I \murec_p(A) = \mathbb E[\mu_\phi(A)] = \mathbb E[ \One[(U,\phi(U)) \in A] ],\]
where in the first expectation, only $\phi$ is random but, in the second one,
both $U$ and $\phi$ are random, independent from each other. This proves the lemma.}
\end{proof}
\begin{rema}
A similar statement (proved in the same way) for the Brownian separable permuton is given in \cite[Lemma 6.1]{maazoun2020BrownianPermuton},
and serves as a basis of Maazoun's computation of the intensity measure of the Brownian separable permuton.
\end{rema}

\subsubsection{Distributional equations}
To go further, we 
consider the following map $\Psi_p$ from the set $\mathcal M_1([0,1])$
of probability measures on $[0,1]$ to itself.
If $\nu$ is a probability measure on $[0,1]$, we let $X$ be a r.v.~ with distribution $\nu$,
and set 
\[Y= V \cdot B  + (1-V) \cdot X,\]
where $B$ is a Bernoulli variable with parameter $p$,
$V$ is a uniform random variable in $[0,1]$, 
and the two are independent from each other and from $X$.
 We then define $\Psi_p(\nu)$ as the distribution of $Y$.
 \begin{lemm}
  The equation $\nu=\Psi_p(\nu)$ has a unique solution in $\mathcal M_1([0,1])$.
  \label{lem:Xp}
\end{lemm}
This solution will be denoted $\nu_p$ from now on.
\begin{proof}
  We will prove that $\Psi_p$ is a contracting map from $\mathcal M_1([0,1])$
  to itself, when we equip it with the first Wasserstein metric.
  Namely, we let $\nu_1$ and $\nu_2$ be two probability measures of $[0,1]$
  and we shall prove that
  \begin{equation}\label{eq:Psi_Contraction}
    d_{W,1}\big(\Psi_p(\nu_1),\Psi_p(\nu_2)\big) \le \tfrac12 d_{W,1}\big(\nu_1,\nu_2\big). 
  \end{equation}
  Fix $\eps>0$ and choose r.v.~$(X_1,X_2)$ on the same probability space
  such that
  \[ \mathbb E\big[\, |X_1-X_2|\, \big] \le  d_{W,1}\big(\nu_1,\nu_2\big) + \eps.\]
  We then take a Bernoulli variable $B$ with parameter $p$ 
  and a  uniform random variable $V$ in $[0,1]$, independent from each other and from $(X_1,X_2)$.  
  We set 
  \[Y_1= V \cdot B + (1-V) \cdot X_1,\qquad Y_2= V \cdot B + (1-V)\cdot X_2.\]
  Note that the variables $V$ and $B$ used to define $Y_1$ and $Y_2$ are the same.
  We have
  \[
    \mathbb E\big[\, |Y_1-Y_2|\, \big] \le  \mathbb E\big[\,|1-V| \cdot |X_1-X_2|\, \big] 
    \le   \tfrac12 \mathbb E\big[\, |X_1-X_2|\, \big] \le \tfrac12 d_{W,1}\big(\nu_1,\nu_2\big) +\tfrac12 \eps.
    \]
  The random variables $Y_1$ and $Y_2$ are defined on the same probability space,
  and have distributions $\Psi_p(\nu_1)$ and $\Psi_p(\nu_2)$ respectively.
  By definition of the Wasserstein distance, we have
  \[d_{W,1}\big(\Psi_p(\nu_1),\Psi_p(\nu_2)\big) \le \mathbb E\big[\, |Y_1-Y_2|\, \big]\le \tfrac12 d_{W,1}\big(\nu_1,\nu_2\big) +\tfrac12 \eps.
\]
Since this holds for any $\eps>0$, we have proved Eq.~\eqref{eq:Psi_Contraction}.
We conclude with Banach fixed point theorem that $\Psi_p$ has exactly one fixed point,
concluding the proof of the lemma.
\end{proof}

\begin{prop}
\label{prop:intensity_UXp}
The intensity measure $I \murec_p$ of the recursive separable permuton
is the distribution of
\[ (U, \, U X_p +(1-U) X_{1-p}),\]
where $U$, $X_p$ and $X_{1-p}$ are independent r.v.~with distribution
$\Leb([0,1])$, $\nu_p$ and $\nu_{1-p}$ distribution respectively.
\end{prop}
\begin{proof}
  From Lemma~\ref{lem:Mickael}, $I \murec_p$ is the distribution of $(U,\phi(U))$,
  where $U$ is a uniform random variable in $[0,1]$ and $\phi$ the random mapping
  constructed in Section~\ref{sec:construction_permuton} .
  We write $\phi(U)=Y_1+Y_2$,
  with
  \[\begin{cases}
    Y_1 = \Leb\big( \{y : (y<U) \wedge (y \prec U) \big);\\
    Y_2 = \Leb\big( \{y : (y>U) \wedge (y \prec U) \big);
  \end{cases}\]
In each of these equations, the first comparison is for the natural order on $[0,1]$, while
  the second is for the random order $\prec$.

  We shall describe the distribution of $Y_2$, conditionally on $U$.
  Let $V_1$ be the {\em first} $U_i$ larger than $U$, \enquote{first} meaning here
   the one with smallest index.
  We then denote $V_2$ to be the first $U_i$ between $U$ and $V_1$ 
  and define $V_3, V_4 \ldots$ similarly.
  Clearly, setting for convenience $V_0=1$, we have
  \[ (V_{j+1}-U)_{j \ge 0} \stackrel{d}= \big( T_{j+1} \, (V_j-U)\big)_{j \ge 0},   \]
  where $(T_1,T_2,\cdots)$ is a sequence of i.i.d.~uniform random variable in $[0,1]$,
  independent from $U$.
  This implies
  \[ (V_{j})_{j \ge 0} \stackrel{d}= \big( U+(1-U)T_1 \cdots T_j \big)_{j \ge 0}.\]
  Each $V_j$ inherits a sign, i.e.~if $V_j=U_i$, we set $\varSigma_j=S_i$.
  By construction, if $\varSigma_j=\ominus$, every $y$ in the interval $[V_{j},V_{j-1}]$ satisfies
  $y \prec U$. On the other hand, $\varSigma_j=\oplus$, 
  every $y$ in the interval $[V_{j},V_{j-1}]$ satisfies $y \succ U$. Hence
  \[Y_2=\sum_{j \ge 1: \ \varSigma_j=\ominus} (V_{j-1}-V_j) = 
  (1-U) \sum_{j \ge 1: \ \varSigma_j=\ominus} T_1 \cdots T_{j-1} (1 -T_j).\]
  Setting $\tilde Y_2=Y_2/(1-U)$ and interpreting $\oplus$ as $1$ and $\ominus$ as $0$,
  we have
  \[ \tilde Y_2 = (1-\varSigma_1) (1-T_1) + T_1 \left(\sum_{j \ge 2: \ \varSigma_j=\ominus} T_2 \cdots (1-T_j) \right).\]
  Note that $ \tilde Y_2$ is independent from $U$.
  Moreover, the variable $1-\varSigma_j$ is a Bernoulli r.v.~of parameter $1-p$, while $T_1$
  is uniform in $[0,1]$.
  Finally, the sum in parentheses has the same distribution as $\tilde Y_1$,
  and is independent from $\varSigma_1$ and $T_1$.
  This shows
  that the distribution of $\tilde Y_2$ is a fixed point of $\Psi_{1-p}$.
  From Lemma~\ref{lem:Xp}, $\tilde Y_2$ has distribution $\nu_{1-p}$.

  With similar arguments, one can show that $Y_1=U \tilde Y_1$, where $\tilde Y_1$
  has distribution $\nu_p$ and is independent from $U$ (and from $\tilde Y_2$).
  This ends the proof of the proposition.
\end{proof}

\subsubsection{Explicit formulas for densities}
\begin{lemm}
  The unique solution $\nu_p$ of the equation $\nu_p=\Psi(\nu_p)$ is
  the beta distribution of parameters $(p,1-p)$. Explicitly it is given by
  \[\nu_p(dx)=\frac{1}{\Gamma(p)\, \Gamma(1-p)} x^{p-1} (1-x)^{-p} \, dx,\]
  where $\Gamma$ is the usual gamma function.
  \label{lem:density_Xp}
\end{lemm}
\begin{proof}
  Let $X$ be a random variable with distribution $\mathrm{Beta}(p,1-p)$,
and set $Y=V \cdot B+(1-V)\, X$, where $B$ and $V$ are as above.
We want to show that $Y\stackrel{d}= X$, which would imply the distribution
of $X$ is a fixed point of $\Psi$ as wanted.

Let $f$ be a continuous function on $[0,1]$. Setting $Z=\Gamma(p)\, \Gamma(1-p)$,
 we have
\begin{align}
  \mathbb E[f(Y)]&=\frac{p}{Z} \int_{[0,1]^2} f\big( v+(1-v)x \big)  x^{p-1} (1-x)^{-p}\, dvdx \nonumber\\
  &\qquad \qquad\qquad
  + \frac{1-p}{Z} \int_{[0,1]^2} f\big( (1-v)x \big) x^{p-1} (1-x)^{-p}\, dvdx \nonumber\\
  &\hspace{-3mm} = \frac{p}{Z}  \int_0^1 f(u)  \left(\int_0^u x^{p-1} (1-x)^{-p-1} dx\right) du \nonumber\\
  &\qquad \qquad\qquad
  + \frac{1-p}{Z} \int_0^1 f(u) \left(\int_u^1 x^{p-2} (1-x)^{-p} dx\right) du.
  \label{eq:Efy}
\end{align}
(In the first integral, we have performed the change of variables $u=v+(1-v)x$,
yielding $dv=\frac{dx}{1-x}$;
in the second, we set $u=(1-v)x$, yielding $dv=\frac{du}{x}$.)

We claim that there exists a constant $A$ in $\mathbb R$,
such that, for every $u$ in $(0,1)$,
\[ p \int_0^u x^{p-1} (1-x)^{-p-1} dx +(1-p) \int_u^1 x^{p-2} (1-x)^{-p} dx = u^{p-1}(1-u)^{-p} +A. \]
Indeed, one checks easily that both sides have the same derivative.
With this equality in hand, \eqref{eq:Efy} rewrites as
\begin{align*}
  \mathbb E[f(Y)]&=\frac{1}Z \int_0^1 f(u) u^{p-1}(1-u)^{p-1}du + \frac{A}{Z}\int_0^1 f(u) du \\
  &= \mathbb E[f(X)] + \frac{A}{Z}\int_0^1 f(u) du.
\end{align*}
Choosing for $f$ the function constant equal to 1 shows that necessarily $A=0$.
Thus we have that $\mathbb E[f(Y)]=\mathbb E[f(X)]$ for any continuous function $f$ of $[0,1]$,
implying that $X$ and $Y$ have the same distribution.
This ends the proof of the lemma.
\end{proof}
Propostion~\ref{prop:intensity_UXp} and Lemma~\ref{lem:density_Xp} imply
Proposition~\ref{prop:intensity_beta}. It remains to prove Corollary~\ref{corol:intensity_density}.

\begin{proof}[Proof of Corollary~\ref{corol:intensity_density}.]
Let $f$ be a continuous function from $[0,1]^2$ to $\mathbb R$.
From Proposition~\ref{prop:intensity_beta}, we have
\begin{multline*}
 {\Gamma(p)^2 \Gamma(1-p)^2}
 \int f(x,y)  \, I \murec_p(dx,dy) \\
=  \int_{[0,1]^3} f\big(u,ua+(1-u)b\big) \, du \, a^{p-1}(1-a)^{-p} da \, b^{-p} (1-b)^{p-1} db
\end{multline*}
We perform the change of variable
\[
x=u; \quad
y=ua+(1-u)b;\quad z=ua
\]
This maps bijectively the $(u,a,b)$-domain $[0,1]^3$ to the set
\[\{(x,y,z):\, \max(x+y-1,0) \le z \le \min(x,y) \}.\]
The Jacobian matrix of the transformation is 
\[J:=  \frac{\partial (x,y,z)}{\partial (u,a,b)} =
\begin{pmatrix} 1 & 0 &0 \\
a-b
  & u& 1-u \\ a &u& 0
\end{pmatrix},\]
whose determinant satisfies $|\det(J)|=u(1-u)=x(1-x)$.
Therefore we have
\begin{multline*}
 {\Gamma(p)^2 \Gamma(1-p)^2}  \int f(x,y) \, I \murec_p(dx,dy) 
=  \int_{[0,1]^2} \Bigg[ f(x,y) \cdot \\
  \cdot \left( \int_{\max(x+y-1,0)}^{\min(x,y)} \, \big(\tfrac{z}x \big)^{p-1}\big(1-\tfrac{z}x\big)^{-p} \, \big(\tfrac{y-z}{1-x}\big)^{-p} \big(1-\tfrac{y-z}{1-x}\big)^{p-1}  dz \right) \frac{dx\, dy}{x(1-x)} \bigg].
\end{multline*}
After elementary simplifications, we get
\begin{multline*}
 {\Gamma(p)^2 \Gamma(1-p)^2} 
 \int f(x,y)  \, I \murec_p(dx,dy) 
= \int_{[0,1]^2} \Bigg[ f(x,y) \cdot \\
\left(\int_{\max(x+y-1,0)}^{\min(x,y)} \frac{dz}{z^{1-p}(x-z)^p(y-z)^p(1-x-y+z)^{1-p}}  \right) dx\, dy \bigg].
\end{multline*}
This proves Corollary~\ref{corol:intensity_density}.
\end{proof}

\subsection{Mutual singularity of the  limiting permutons}
\label{ssec:Singular_Distributions}
In this section, we prove Proposition~\ref{prop:singularity} in two independent steps. First,
we prove the singularity of separable Brownian or recursive separable permutons 
associated with different values $p$ and $q$ of the parameter.
Then we compare specifically Brownian and recursive separable
permutons $\murec_p$ and $\muBr_p$, associated with the same value $p$ of the parameter.

Before starting the proof, let us recall, more formally than in Remark~\ref{rem:Brownian}, the construction of $\muBr_p$,
as given in \cite{maazoun2020BrownianPermuton}.
We start with a Brownian excursion $\exc$ on $[0,1]$
and a sequence $(S_m)_{m \in \Min(\exc)}$ of signs indexed by local minima of $\exc$.
Conditionally on $\exc$, the variables $(S_m)_{m \in \Min(\exc)}$ are i.i.d.~with distribution
\[\mbb P(S_m=\oplus)=p=1-\mbb P(S_m=\ominus).\]
Given such a sequence we define a partial order $\prec_{\text{Br}}$ as follows: for $x<y$ in $[0,1]$,
we let $m$ be the position of the minimum of $\exc$ on the interval $[x,y]$ 
and set
\[ \begin{cases}
  x \prec_{\text{Br}} y &\text{ if }S_m=\oplus;\\
y \prec_{\text{Br}} x &\text{ if }S_m=\ominus.
\end{cases}\]
Note that, if $m \in \{x,y\}$, then $m$ might not be a local minimum, in which case
$S_m$ is ill-defined and $x$ and $y$ are incomparable by convention.
This happens only for a measure $0$ subset of pairs $(x,y)$ (w.r.t.~Lebesgue measure).

The rest of the construction is then similar to that of the recursive separable permuton:
we define
 \[\phi_{\text{Br}}(x)= \Leb(\{y \in [ 0, 1 ] : y \prec_{\text{Br}} x\}),\]
and let $\muBr_p$ be the push-forward of the Lebesgue measure on $[0,1]$
by the map $x\mapsto (x,\phi_{\text{Br}}(x))$.

\subsubsection{Comparing permutons with different values of $p$}
We recall from Section~\ref{ssec:pattern_densities_intro}
that given a permuton $\mu$ and an integer $n \ge 1$,
we can define a random permutation $\Sample(\mu,n)$ 
by sampling independent points according to $\mu$.
Also, for a permutation $\pi$, we let $\des(\pi)$ be its number of descents.
We start with a lemma.
\begin{lemm}
\label{lem:number-of-descents}
Let $p$ be in $(0,1)$. Then the random variable
\[ D_n=D_n(\murec_p):= \frac1{n-1} \mbb E\big[ \des(\Sample(\murec_p,n) ) \big| \, \murec_p \big]. \]
converges to $1-p$ a.s. 
Moreover, the same holds replacing $\murec_p$ by $\muBr_p$.
\end{lemm}
\begin{proof}
We have
\[ \mbb E(D_n) = \frac1{n-1} \mbb E\big[ \des(\Sample(\murec_p,n) ) \big].\]
From the proof of Proposition~\ref{prop:expected_densities_murec}, 
the permutation $\Sample(\murec_p,n)$
has the same distribution as the permutation represented by a uniform random
increasing binary tree $T_n$, where each internal node is decorated with $\oplus$
independently with probability $p$.
As already observed, if $\pi$ is encoded by a binary decorated tree $T$,
then $\des(\pi)$ is the number of $\ominus$ signs in $T$.
Hence $\des(\Sample(\murec_p,n) )$ is the number of minus signs in $T_n$,
and its law is that of a binomial random variable $\Bin(n-1,1-p)$.
We deduce that $\mbb E(D_n) =1-p$.

To get a.s.~convergence, we consider the fourth centered moment of $D_n$.
We have, using Jensen's inequality for conditional expectation,
\begin{align*}
 \mbb E\big[ (D_n - (1-p))^4 \big] &= 
 \mbb E\bigg[ \mbb E \Big( \tfrac1{n-1} \des(\Sample(\murec_p,n) ) -(1-p)  \big|  \, \murec_p \Big)^4 \bigg]\\
& \le   \mbb E\bigg[ \mbb E \Big[\Big( \tfrac1{n-1}\des(\Sample(\murec_p,n) ) -(1-p) \Big)^4 \Big|  \, \murec_p \Big] \bigg]\\
& = \mbb E\bigg[ \Big( \tfrac1{n-1}\des(\Sample(\murec_p,n) ) -(1-p) \Big)^4  \bigg].
\end{align*}
The expectation in the right-hand side is the centered fourth moment 
 of the average of $n-1$ independent Bernoulli random variables of parameter $1-p$,
 which is known to behave as $O(n^{-2})$.
We get that
\[ \mbb E\big[ (D_n -  (1-p) )^4 \big] =O(n^{-2}),\]
and hence it is a summable quantity. By a classical application of Borel--Cantelli lemma,
this implies that $D_n$ converges a.s.~to $1-p$.

The proof for the Brownian separable permuton is similar.
Indeed, from \cite[Definition 2]{maazoun2020BrownianPermuton}, 
we know that the quantity $\des(\Sample(\muBr_p,n))$ is also distributed
as a binomial random variable $\Bin(n-1,1-p)$.
\end{proof}

\begin{coro}
If $p \ne q$ are in $(0,1)$, then $\murec_p$ and $\murec_q$ are mutually singular.
The same holds, replacing either $\murec_p$ or $\murec_q$, or both,
by $\muBr_p$ or $\muBr_q$.
\end{coro}
\begin{proof}
Let $E_p$ be the set of permutons $\mu$
such that 
\[D_n(\mu)=  \frac1{n-1} \mbb E\big[ \des(\Sample(\mu,n) ) \big] \]
converges to $1-p$ (we do not take a conditional expectation here, since $\mu$ is deterministic).
By Lemma~\ref{lem:number-of-descents}, $\murec_p$ and $\muBr_p$ belongs to $E_p$ a.s.
Since $E_p$ and $E_q$ are disjoint for $p \ne q$, the corollary follows.
\end{proof}

\subsubsection{Comparing the Brownian and recursive separable permutons
with the same parameter}
\label{ssec:singular_different_parameters}
We now want to prove that $\murec_p$ and $\muBr_p$ are singular.
Since both $D_n(\murec_p)$ and $D_n(\muBr_p)$ converge to the same value $1-p$ a.s.,
 we need to find another distinguishing feature.
We will prove that in the Brownian separable permuton, all corners of the square a.s.~carry
some mass, which is not the case for the recursive separable permuton.

To this end, for $\eps>0$, we introduce the following events
(we recall that $\mathcal P$ is the set of permutons):
\begin{align*}
\BL_\eps&=\{ \mu \in \mathcal P: \mu([0,\eps]^2) >0 \} \\
\TL_\eps&=\{ \mu \in \mathcal P: \mu([0,\eps]\times [1-\eps,1]) >0 \}.
\end{align*}
In words, $\BL_\eps$ is the set of permutons that have some mass in
an $\eps$-neighbourhood of the bottom-left corner;
$\TL_\eps$ is the same using the top-left corner.

We now prove two lemmas, showing that the Brownian and recursive separable permutons
behave differently with respect to these events.
We start with the recursive separable permuton.
\begin{lemm}
\label{lem:corners_recursive}
Fix $p \in (0,1)$ and $\eps >0$. We have
\[\mbb P\big(\murec_p \in  \BL_\eps \cap \TL_\eps \big) \le 2\eps.\]
\end{lemm}
\begin{proof}
We recall the construction of $\murec_p$ from Section~\ref{sec:construction_permuton} ,
and use the notation introduced there.

We temporarily assume that $U_1$ is in $(\eps,1-\eps)$ and that $S_1=\oplus$.
We consider some $x<\eps$. 
For any $y \ge U_1$, we have $U_1 \in (x,y]$, and therefore $i_{x,y}=1$
with the notation of Section~\ref{sec:construction_permuton} . Since $x<y$ and $S_1=\oplus$,
this implies $x \prec y$. Taking the contraposition, $y \prec x$ implies $y < U_1$.
Therefore,
\[\phi(x)= \Leb(\{y \in [ 0, 1 ] : y \prec x\}) \le \Leb(\{y \in [ 0, 1 ] : y <U_1\} \le U_1 < 1-\eps.\]
Thus there does not exists $x$ such that $(x,\phi(x)) \in [0,\eps] \times [1-\eps,1]$.
Consequently, under the assumptions $U_1 \in (\eps,1-\eps)$ and $S_1=\oplus$,
we have
\[\murec_p([0,\eps]\times [1-\eps,1]) =0,\text{ or equivalenlty, } \murec_p \notin \TL_\eps.\]
Similarly, one can prove that, if $U_1$ is in $(\eps,1-\eps)$ and $S_1=\ominus$,
then $\murec_p \notin \BL_\eps$.
Therefore we have
\[\mbb P\big(\murec_p \in  \BL_\eps \cap \TL_\eps\big) \le \mbb P\big( U_1 \notin (\eps,1-\eps) \big) \le 2\eps. \qedhere \]
\end{proof}

Considering the Brownian separable permuton instead,
we have the following.
\begin{lemm}
\label{lem:corners_Brownian}
Fix $p \in (0,1)$ and $\eps>0$. We have
\[\mbb P\big(\muBr_p \in  \BL_\eps \big) =1.\]
\end{lemm}
\begin{proof}
As recalled above,
$\muBr_p$ can be constructed starting from a Brownian excursion $\exc$ and a sequence of signs $(S_m)_{m \in \Min(\exc)}$ indexed by the (positions of) the local minima of $\exc$.
We fix a realization of $\exc$ and $\bm S=(S_m)_{m \in \Min(\exc)}$ (and hence of $\muBr_p$);
most quantities below, including the random order $\prec_{\text{Br}}$ and the function $\phi_{\text{Br}}$,
depend implicitly on $\exc$ and $\bm S$.
Our first goal is to find a local minimum $m_0$ of $\exc$,
such that  $S_{m_0}=\oplus$, $m_0<\eps/2$ and $\exc(m_0)<\min_{t \in [m_0,1-\eps/2]} \exc(t)$.
Since the proof involves quite a bit of notation, we
illustrate it on Fig.~\ref{fig:corner_Brownian}.
\begin{figure}
\[\includegraphics[height=35mm]{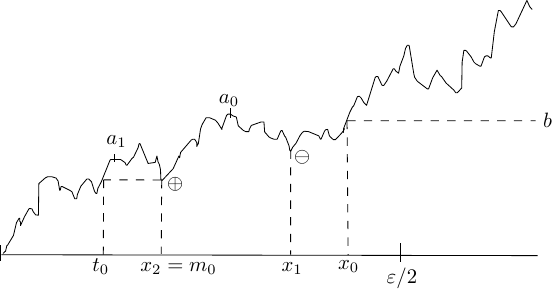}\]
\caption{Illustration of the notation involved in the proof of Lemma~\ref{lem:corners_Brownian}. 
For readability, we have only represented an initial segment of the Brownian excursion $\exc$.
In the picture, we have $S_{x_1}=\ominus$ so that
we have to find another candidate $x_2$ as explained in the proof.
This time we have $S_{x_2}=\oplus$ and we set $m_0=x_2$.}
\label{fig:corner_Brownian}
\end{figure}

We first let $b=\min_{t \in [\eps/2,1-\eps/2]} \exc(t)$
and $x_0=\sup\{t<\eps/2: \exc(t)=m\}$.
A.s., we can then find $a_0<x_0$ such that $\exc$ reaches its minimum on $[a_0,x_0]$
somewhere in the interior of this interval 
(if it was not the case, $\exc$ would be increasing on an initial segment $[0,\delta]$ for some $\delta$, 
which is known to happen with probability $0$).
Let $x_1$ be the point where $\exc$ is minimal on $[a_0,x_0]$.
This is a local minimum, and therefore carries a sign $S_{x_1}$. If $S_{x_1}=\oplus$,
we define $m_0=x_1$, and we verify easily that it satisfies the desired properties.

Otherwise, we iterate the process: a.s., we can then find $a_1<x_1$ such that $\exc$ reaches its minimum on $[a_1;x_1]$
somewhere in the interior of this interval.
We call $x_2$ the point where this minimum is reached. If $S_{x_2}=\oplus$,
we define $m_0=x_2$. If not, we iterate another time.
Doing so, we will construct a sequence of local minima $x_1,x_2,\dots$, and since the associated signs are i.i.d.~and equal to $\oplus$ with positive probability, we will eventually find $x_{i}$ with $S_{x_i}=\oplus$.
Then we set $m_0=x_i$, and verify easily that $m_0$ satisfies the desired properties.

Having found $m_0$, we look for the last time $t_0$ before $m_0$ with $\exc(t_0)=\exc(m_0)$. By construction, if we take $x$ in $(t_0,m_0)$ and $y$ in $(m_0,1-\eps/2)$ the minimum of $\exc$ in the interval $[x,y]$ is reached in $m_0$.
Since $S_{m_0}=\oplus$, we have $x \prec_{\text{Br}} y$, 
where $\prec_{\text{Br}}$ is the order appearing in the construction of the Brownian separable permuton.
Therefore, for $x$ in $(t_0,m_0)$, we have
\[\phi_{\text{Br}}(x)= \Leb(\{y \in [ 0, 1 ] : y \prec_{\text{Br}} x\}) \le 1-\Leb( (m_0,1-\eps/2) ) =m_0 +\eps/2 <\eps.\]
Letting $U$ be a uniform random variable in $[0,1]$, we have, a.s., 
\[\muBr_p([0,\eps]^2) = \mbb P\big( (U,\phi_{\text{Br}}(U)) \in [0,\eps]^2 | (\exc,\bm S) \big)
\ge \mbb P\big( U \in (t_0,m_0) | (\exc,\bm S) \big) =m_0-t_0>0;\]
(recall that $m_0$ and $t_0$ depends on $(\exc,\bm S)$ and thus are random variables themselves).
This proves the lemma.
\end{proof}

\begin{coro}
Fix $p$ in $(0,1)$. Then the distributions of $\murec_p$ and $\muBr_p$
are mutually singular.
\end{coro}
\begin{proof}
We consider the decreasing intersection
\[E:= \bigcap_{\eps >0} ( \BL_\eps \cap \TL_\eps ). \]
From Lemma~\ref{lem:corners_recursive}, we know that $\mbb P\big(\murec_p \in E)=0$.
On the other hand, Lemma~\ref{lem:corners_Brownian} tells us that, for any $\eps>0$, we have
\[\mbb P\big(\muBr_p \in  \BL_\eps \big) =1.\]
By symmetry, the same holds replacing $\BL_\eps$ by $\TL_\eps$,
and thus, for any $\eps>0$, 
\[\mbb P\big(\muBr_p \in  \BL_\eps \cap \TL_\eps \big) =1.\]
Consequently, $\mbb P\big(\murec_p \in E)=1$, proving the corollary.
\end{proof}

\begin{rema}
The difference between recursive and Brownian separable permutons
which we exhibited in the proof can be observed on simulations.
Fig.~\ref{fig:compare_Brownian_recursive} shows simulations of a random recursive 
separable permutation with parameter $p=1/2$ (on the left) and of a uniform random separable permutation (on the right), both of size 1000. From the result of this paper and of \cite{bassino2018BrownianSeparable}, the corresponding limiting permutons
are the recursive and Brownian separable permutons, respectively, each time of parameter $p=1/2$.
We see that in the Brownian case (picture on the right),
there are some points relatively close to each one of the four corners of the square,
 which is not the case in the recursive case (picture on the left).
 
 {An informal explanation of that can also be given at the discrete level.
 As we have seen, recursive separable permutations of size $n$ are
 associated with uniform random {\em labelled increasing} binary trees (Section~\ref{ssec:expected_pattern_densities}).
 It is known that in such random trees,
  the two subtrees attached to the root splits both have macroscopic size
 (this follows, for example, from the correspondence between increasing trees and permutations; see, {\em e.g.}, \cite[Example II.7]{flajolet2009analytic}).
 Hence $\sigma^{(n),p}=\tau \oplus \rho$ or $\sigma^{(n),p}=\tau \ominus \rho$,
 for some permutations $\tau$ and $\rho$ of macroscopic sizes
 (this can also directly be seen on Proposition~\ref{prop:self_similarity_discrete}).
 This explains why some corners of the permutation diagram of $\sigma^{(n),p}$ are empty.}

{ On the other hand, uniform random separable permutations of size $n$ are                              
 associated with a uniform random Schroder tree (see \cite{bassino2018BrownianSeparable}).
 In such a random tree, 
among the subtree attached to the root,
a single one of them contains most of the vertices of the trees, 
 the other ones having typically size $O(1)$; 
 see, {\em e.g.}, the local limit results given in \cite{janson2012simply}.
 Such unbalanced splits are repeated many times, with signs either $\oplus$ and $\ominus$, before a macroscopic split arises,
 and this explains why we find points in the permutation diagram near all corners of the square.}
\end{rema}
\begin{figure}
\[ \includegraphics[height=45mm]{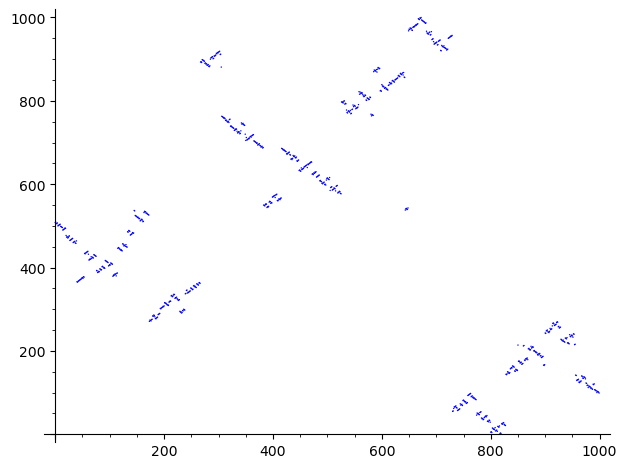} \quad  \includegraphics[height=45mm]{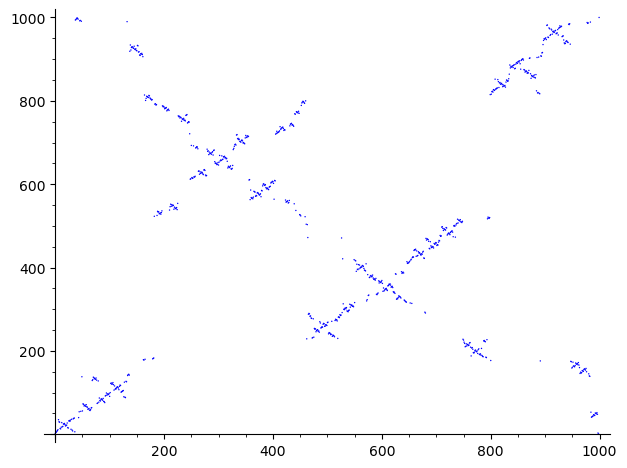}\]
\caption{Diagram of random separable permutations of 1000:
on the left, the permutation has been generated via the recursive algorithm given in Section~\ref{ssec:model}; on the right it is a uniform random separable permutation of 
size 1000 (sampled via the so-called recursive method).}
\label{fig:compare_Brownian_recursive}
\end{figure}

\section*{Acknowledgements}
An important part of this work was conducted while the second author (KRL)
was a postdoctoral fellow at Universit\'e de Lorraine {funded} by the I-Site Lorraine Universit\'e d'Excellence (LUE). Both authors are grateful to LUE for this opportunity.

The authors would like to thank Cyril Marzouk for a discussion about recursive
models of tree-like structures, and in particular for pointing out reference \cite{curien2009recursive}.
{They are also grateful to anonymous referees for their suggestions
to improve the presentation of the paper.
One of the referee suggested in particular the current proof
of Proposition~\ref{prop characterization pushforward convergence},
which is shorter and more probabilistic than the one in the first arXiv version of the paper.}

Finally, simulations of Figures \ref{fig:simu_increasing_separable_several_sizes}, \ref{fig:recursive_cographs}, \ref{fig:graph_phik}, \ref{fig:lambda} and \ref{fig:compare_Brownian_recursive},
as well as the generation of data for Figure~\ref{fig:discrete_intensity},
were done with the computer algebra software SageMath~\cite{sagemath}.
The authors are grateful to its developers for their wonderful job.

\end{document}